\documentclass[11pt]{article}
\usepackage[utf8]{inputenc}
\usepackage{natbib}
\usepackage{amsmath}
\usepackage{amsthm}
\usepackage{amssymb}
\usepackage{amsfonts}
\usepackage{dutchcal}
\usepackage{graphicx}
\usepackage[hidelinks=true]{hyperref}
\usepackage{lscape}
\usepackage{authblk}

\hypersetup{urlcolor=red, citecolor=blue, colorlinks=true, linkcolor=blue}

\usepackage[margin=0.9in]{geometry}

\newtheorem{theorem}{Theorem}
\newtheorem{example}{Example}
\newtheorem{definition}{Definition}
\newtheorem{lemma}{Lemma}

\newtheorem{remark}{Remark}
\DeclareMathOperator\erfc{erfc}

\begin{document}

\title{The Gauss Hypergeometric Covariance Kernel for Modeling Second-Order Stationary Random Fields in Euclidean Spaces: its Compact Support, Properties and Spectral Representation}





\author[1,2]{Xavier Emery}
\affil[1]{\small{Department of Mining Engineering, University of Chile, Avenida Beauchef 850, Santiago 8370448, Chile.}}
\affil[2]{Advanced Mining Technology Center, University of Chile, Avenida Beauchef 850, Santiago 8370448, Chile.}

\author[3]{Alfredo Alegr\'ia\footnote{Corresponding author. Email: alfredo.alegria@usm.cl}}
\affil[3]{Departamento de Matem\'atica, Universidad T{\'e}cnica Federico Santa Mar{\'i}a, Valpara{\'i}so, Chile.}

\maketitle

\begin{abstract}
This paper presents a parametric family of compactly-supported positive semidefinite kernels aimed to model the covariance structure of second-order stationary isotropic random fields defined in the $d$-dimensional Euclidean space. Both the covariance and its spectral density have an analytic expression involving the hypergeometric functions ${}_2F_1$ and ${}_1F_2$, respectively, and four real-valued parameters related to the correlation range, smoothness and shape of the covariance. The presented hypergeometric kernel family contains, as special cases, the spherical, cubic, penta, Askey, generalized Wendland and truncated power covariances and, as asymptotic cases, the Mat\'ern, Laguerre, Tricomi, incomplete gamma and Gaussian covariances, among others. The parameter space of the univariate hypergeometric kernel is identified and its functional properties --- continuity, smoothness, transitive upscaling (mont\'ee) and downscaling (descente) --- are examined. Several sets of sufficient conditions are also derived to obtain valid stationary bivariate and multivariate covariance kernels, characterized by four matrix-valued parameters. Such kernels turn out to be versatile, insofar as the direct and cross-covariances do not necessarily have the same shapes, correlation ranges or behaviors at short scale, thus associated with vector random fields whose components are cross-correlated but have different spatial structures.

\vspace{0.5cm}


\end{abstract}

\section{Introduction}

Geostatistical techniques such as kriging or conditional simulation are widely used to interpolate regionalized data, in order to address spatial prediction problems or to quantify uncertainty at locations without data \citep{chiles2009geostatistics}. These techniques rely on a modeling of the spatial correlation structure of one or more regionalized variables, viewed as realizations of as many spatial random fields. Application domains include natural (mineral, oil and gas) resources assessment, groundwater hydrology, soil and environmental sciences, among many others, where it is not uncommon to work with up to a dozen variables \citep{ahmed, emery2021, hohn, webster}. This motivates the need for univariate and multivariate covariance (positive semidefinite) kernels that allow a flexible parameterization of the relevant properties such as the correlation range or the short-scale regularity. In practical applications, the random fields under study are often assumed to be second-order stationary, i.e., their first- and second-order moments (expectation and covariance) exist and are invariant under spatial translation \citep{chiles2009geostatistics, cressie, wackernagel}. The stationarity assumption is made throughout this work, which implies that the covariance kernel for two input vectors $\boldsymbol{s}$ and $\boldsymbol{s}^\prime$ is actually a function of the separation $\boldsymbol{h}=\boldsymbol{s}-\boldsymbol{s}^\prime$ between these vectors (here $\boldsymbol{s}$ and $\boldsymbol{s}^\prime$ are elements of the $d$-dimensional Euclidean space $\mathbb{R}^d$) and that its Fourier transform (spectral density of the covariance kernel) is also a function of a single vectorial argument $\boldsymbol{u} \in \mathbb{R}^d$ \citep{chiles2009geostatistics, wackernagel}.

Many parametric families of stationary covariance kernels have been proposed in the past decades, the most widespread being the Mat\'ern kernel \citep{Matern} that allows controlling the behavior of the covariance at the origin. This kernel has been extended to the multivariate case \citep{Apanasovich, Gneiting2010}, offering a more flexible parameterization than the traditional linear model of coregionalization \citep{wackernagel}, but still suffers from restrictive conditions on its parameters to be a valid coregionalization model. 

Compactly-supported covariance kernels possess nice computational properties that make them of particular interest for applications, insofar as they are suitable to likelihood-based inference and kriging in the presence of large data sets when combined with algorithms for solving sparse systems of linear equations \citep{furrer2006covariance,kaufman2008covariance}, and to specific simulation algorithms such as circulant-embedding and FFT-based approaches \citep{chiles2009geostatistics, Dietrich, pardo, wood1994simulation}. However, although many families of such kernels have been elaborated for the modeling of univariate random fields, such as the spherical, cubic, Askey, Wendland and generalized Wendland families \citep{Askey, chiles2009geostatistics, Hubbert, matheron1965variables, Wendland}, so far there is still a lack of flexible families of multivariate compactly-supported covariance kernels, with few notable exceptions \citep{porcu2013, daley}.

This paper deals with the design of a wide parametric family of compactly-supported covariance kernels for second-order stationary univariate and multivariate random fields in $\mathbb{R}^d$, and with the determination of their parameter space, functional properties, spectral representations and asymptotic behavior. The intended family of covariance kernels will contain all the above-mentioned kernels, as well as the Mat\'ern kernel as an asymptotic case. Estimating the kernel parameters from a set of experimental data, comparing estimation approaches or examining the impact of the parameters in spatial prediction or simulation outputs are out of the scope of this paper and are left for future research.  The outline is the following: Section \ref{univariate} presents the univariate kernel and its properties. This kernel is then extended to multivariate random fields in Section \ref{multivariate} and to specific bivariate random fields in Section \ref{bivariate}. Conclusions follow in Section \ref{conclusions}, while technical definitions, lemmas and proofs are deferred to Appendices \ref{appendixA} and \ref{appendixB}.

\section{A class of stationary univariate compactly-supported covariance kernels}
\label{univariate}

\subsection{Notation}

For $k, k^\prime \in \mathbb{N}$, $\alpha_1, \hdots, \alpha_k, \beta_1, \hdots, \beta_{k^\prime} \in \mathbb{R}$, the generalized hypergeometric function ${}_kF_{k^\prime}$ in $\mathbb{R}$ is defined by the following power series \citep[formula 16.2.1]{olver2010nist}:
\begin{equation}
\label{pFq}
    {}_kF_{k^\prime}(\alpha_1, \hdots, \alpha_k; \beta_1, \hdots, \beta_{k^\prime}; x) = 1+\sum_{n=1}^{+\infty} \frac{\prod_{i=1}^k \Gamma(\alpha_i+n) \prod_{j=1}^{k^\prime} \Gamma(\beta_j)}{\prod_{i=1}^k \Gamma(\alpha_i) \prod_{j=1}^{k^\prime} \Gamma(\beta_j+n)} \frac{x^n}{n!}, \quad x \in \mathbb{R},
\end{equation}
where $\Gamma$ is Euler's gamma function. The series \eqref{pFq} converges for any $x \in \mathbb{R}$ if $k<k^\prime+1$, for any $x \in ]-1,1[$ if $k=k^\prime+1$ and also for $x=\pm 1$ if $k=k^\prime+1$ and $\sum_{i=1}^k \alpha_i < \sum_{j=1}^{k^\prime} \beta_j$. Specific cases include the confluent hypergeometric limit function ${}_0F_1$, Kummer's confluent hypergeometric function ${}_1F_1$ and Gauss hypergeometric function ${}_2F_1$. 

\subsection{Kernel construction}

Consider the isotropic function $\widetilde{G}_d(\cdot; a,\alpha,\beta,\gamma)$ defined in $\mathbb{R}^d$ by:
\begin{equation}
    \label{spectral_density}
    \begin{split}
    \widetilde{G}_d({\boldsymbol u}; a,\alpha,\beta,\gamma) &= \widetilde{g}_d(\|{\boldsymbol u}\|; a,\alpha,\beta,\gamma) \\
    &=  \zeta_d(a,\alpha,\beta,\gamma) \, {}_1 F_2\left( \alpha;\beta,\gamma; -(\pi a \|{\boldsymbol u}\|)^2 \right), \quad {\boldsymbol u} \in \mathbb{R}^d,
    \end{split}
\end{equation}
where $\|\cdot\|$ denotes the Euclidean norm, $d$ is a positive integer, $(a,\alpha,\beta,\gamma)$ are positive scalar parameters, and $\zeta_d(a,\alpha,\beta,\gamma)$ is a normalization factor that will be determined later. \cite{cho2020rational} proved that $\widetilde{G}_d(\cdot; a,\alpha,\beta,\gamma)$ is nonnegative under the following conditions:
\begin{itemize}
    \item $\alpha > 0$;
    \item $2(\beta-\alpha)(\gamma-\alpha) \geq \alpha$;
    \item $2(\beta+\gamma) \geq 6\alpha + 1$.
\end{itemize}

Hereinafter, $\mathcal{P}_0$ denotes the set of triplets $(\alpha,\beta,\gamma)$ of $\mathbb{R}_+^3$ satisfying these three conditions; note that the last two conditions imply, in particular, that $\beta > \alpha$ and $\gamma > \alpha$.
Under an additional assumption of integrability, $\widetilde{G}_d(\cdot; a,\alpha,\beta,\gamma)$ is the spectral density associated with a stationary isotropic covariance kernel $G_d(\cdot; a,\alpha,\beta,\gamma)$ in $\mathbb{R}^d$. 
Let $g_d(\cdot; a,\alpha,\beta,\gamma): \mathbb{R}_+ \to \mathbb{R}$ denote the radial part of such a covariance kernel: $G_d({\boldsymbol h}; a,\alpha,\beta,\gamma) = g_d(\|{\boldsymbol h}\|; a,\alpha,\beta,\gamma)$ for ${\boldsymbol h} \in \mathbb{R}^d$. Following the scaling conventions used by \cite{steinweiss} to define the Fourier and inverse Fourier transforms, $g_d(\cdot; a,\alpha,\beta,\gamma)$ is the Hankel transform of order $d$ of $\widetilde{g}_d(\cdot; a,\alpha,\beta,\gamma)$, i.e.:
\begin{equation}
\label{hankel}
g_d(r; a,\alpha,\beta,\gamma)  =  \frac{2\pi \zeta_d(a,\alpha,\beta,\gamma)}{r^{\frac{d}{2}-1}} \int_0^{+\infty} \rho^{\frac{d}{2}}  J_{\frac{d}{2}-1}(2\pi\rho r) {}_1 F_2\left(\alpha;\beta,\gamma; -(\pi a \rho)^2\right) \text{d}\rho, \quad r > 0,
\end{equation}
with $J_\mu$ denoting the Bessel function of the first kind of order $\mu$. By using formulae 16.5.2 and 10.16.9 of \cite{olver2010nist}, the generalized hypergeometric function ${}_1 F_2$ can be written as a beta mixture of Bessel functions of the first kind:
\begin{equation}
\label{mixbessel}
\begin{split}
  {}_1 F_2&\left(\alpha;\beta,\gamma; -(\pi a \rho)^2)\right)   \\
  &  =      \frac{\Gamma(\beta)}{\Gamma(\alpha)\Gamma(\beta-\alpha)} \int_0^1 t^{\alpha-1} (1-t)^{\beta-\alpha-1} {}_0 F_1\left(; \gamma;-t(\pi a \rho)^2 \right) \text{d}t \\
  &  =      \frac{\Gamma(\beta)\Gamma(\gamma)}{\Gamma(\alpha)\Gamma(\beta-\alpha)} \int_0^1 t^{\alpha-1} (1-t)^{\beta-\alpha-1} \left( \pi a \rho \sqrt{t} \right)^{1-\gamma}  J_{\gamma-1}\left(2\pi a \rho \sqrt{t} \right)  \text{d}t.\\
\end{split}
\end{equation}

Owing to Fubini's theorem, the radial function \eqref{hankel} is found to be
\begin{equation}
\label{fubini}
\begin{split}
    g_d&(r; a,\alpha,\beta,\gamma)  = \frac{2\pi (\pi a)^{1-\gamma} \Gamma(\beta)\Gamma(\gamma)}{r^{\frac{d}{2}-1} \Gamma(\alpha)\Gamma(\beta-\alpha)}  \zeta_d(a,\alpha,\beta,\gamma)  \\
    & \times 
    \int_0^1 t^{\alpha-1/2-\gamma/2} (1-t)^{\beta-\alpha-1}   \int_0^{+\infty} \rho^{\frac{d}{2} + 1 - \gamma}  J_{\frac{d}{2}-1}(2\pi\rho r)  J_{\gamma-1}\left(2\pi a \rho \sqrt{t} \right)  \text{d}\rho \text{d}t.
\end{split}
\end{equation}

For $\gamma>\frac{d}{2}$, the last integral in \eqref{fubini} is convergent and can be evaluated by using formula 6.575.1 of \cite{grad}:
\begin{equation}
\label{hygeo6}
\begin{split}
 g_d&(r; a,\alpha,\beta,\gamma)  \\
 &=  \begin{cases}
 \frac{\pi^{-\frac{d}{2}}a^{-d}\Gamma(\beta) \Gamma(\gamma) \zeta_d(a,\alpha,\beta,\gamma)}{\Gamma(\alpha) \Gamma(\beta-\alpha) \Gamma(\gamma - \frac{d}{2})}  \int_{\left(\frac{r}{a}\right)^2}^1 t^{\alpha-\gamma} (1-t)^{\beta-\alpha-1} \left( t - \left(\frac{r}{a}\right)^2 \right)^{\gamma - \frac{d}{2} - 1} \text{d}t &  \text{if } 0<r\leq a\\
0  & \text{if } r> a.
\end{cases}  
\end{split}
\end{equation}

The function $g_d(\cdot; a,\alpha,\beta,\gamma)$ so defined can be extended by continuity at $r=0$ if $\alpha>\frac{d}{2}$ (\citealp{grad}, formulae 3.191.3):
\begin{equation*}
    g_d(0; a,\alpha,\beta,\gamma)  = 
\frac{\pi^{-\frac{d}{2}}a^{-d} \Gamma(\alpha-\frac{d}{2}) \Gamma(\beta) \Gamma(\gamma) \zeta_d(a,\alpha,\beta,\gamma)}{\Gamma(\alpha)   \Gamma(\beta - \frac{d}{2}) \Gamma(\gamma - \frac{d}{2})}. 
\end{equation*}

This value is equal to one when considering the following normalization factor:
\begin{equation}
\label{hygeo8}
\zeta_d(a,\alpha,\beta,\gamma) = \frac{\pi^{\frac{d}{2}} a^d \Gamma(\alpha)   \Gamma(\beta - \frac{d}{2}) \Gamma(\gamma - \frac{d}{2})}{\Gamma(\alpha-\frac{d}{2}) \Gamma(\beta) \Gamma(\gamma)}.  
\end{equation}

\subsection{Analytic expressions and parameter space}

By substituting \eqref{hygeo8} in \eqref{spectral_density} and \eqref{hygeo6}, one obtains the following expressions for the spectral density and the covariance kernel:
\begin{equation}
\label{spectraldensity2}
    \widetilde{G}_d({\boldsymbol u}; a,\alpha,\beta,\gamma) =   \frac{\pi^{\frac{d}{2}} a^d \Gamma(\alpha)   \Gamma(\beta - \frac{d}{2}) \Gamma(\gamma - \frac{d}{2})}{\Gamma(\alpha-\frac{d}{2}) \Gamma(\beta) \Gamma(\gamma)} \, {}_1 F_2\left( \alpha;\beta,\gamma; -(\pi a \|{\boldsymbol u}\|)^2 \right), \quad {\boldsymbol u} \in \mathbb{R}^d,
\end{equation}
and
\begin{equation}
\label{hygeo9}
\begin{split}
  G_d&({\boldsymbol h}; a,\alpha,\beta,\gamma)  \\
  &=  \begin{cases}
 \frac{ \Gamma(\beta - \frac{d}{2})}{\Gamma(\alpha-\frac{d}{2}) \Gamma(\beta-\alpha)}  \int_{\left(\frac{\|{\boldsymbol h}\|}{a}\right)^2}^1 t^{\alpha-\gamma} (1-t)^{\beta-\alpha-1} \left( t - \left(\frac{\|{\boldsymbol h}\|}{a}\right)^2 \right)^{\gamma - \frac{d}{2} - 1} \text{d}t &  \text{if } 0\leq \|{\boldsymbol h}\|\leq a\\
0  & \text{if } \|{\boldsymbol h}\|> a.
\end{cases}  
\end{split}
\end{equation}

Hereinafter, $G_d(\cdot; a,\alpha,\beta,\gamma)$ will be referred to as the Gauss hypergeometric covariance, the reason being that it has the following analytic expression, obtained from \eqref{hygeo9} by using formula II.1.4 of \cite{matheron1965variables}:
\begin{equation}
\label{2F1}
\begin{split}
G_d({\boldsymbol h}; a,\alpha,\beta,\gamma)  = & \frac{\Gamma(\beta-\frac{d}{2}) \Gamma(\gamma-\frac{d}{2})}{\Gamma(\beta-\alpha+\gamma-\frac{d}{2}) \Gamma(\alpha-\frac{d}{2})} \left( 1 - \frac{\|{\boldsymbol h}\|^2}{a^2} \right)_+^{\beta - \alpha+\gamma -\frac{d}{2}-1} \\
& \times {}_2 F_1 \left(\beta-\alpha,\gamma-\alpha; \beta-\alpha+\gamma-\frac{d}{2};  \left(1-\frac{\|{\boldsymbol h}\|^2}{a^2}\right)_+ \right), \quad {\boldsymbol h} \in \mathbb{R}^d,
\end{split}
\end{equation}
with $()_+$ denoting the positive part function. A wealth of closed-form expressions can be obtained for specific values of the parameters $\alpha$, $\beta$ and $\gamma$, see examples in forthcoming subsections. 

Also, several algorithms and software libraries are available to accurately compute the confluent hypergeometric limit function ${}_0F_1$ and the Gauss hypergeometric function ${}_2F_1$ \citep{galassi2009gnu, Johansson, Johansson2, pearson}, allowing the numerical calculation of both the covariance \eqref{2F1} and its spectral density \eqref{spectraldensity2}, the latter being written as a beta mixture of ${}_0F_1$ function as in \eqref{mixbessel}. Consequently, the proposed hypergeometric kernel can be used without any difficulty for kriging or for simulation (in the scope of Gaussian random fields) based on matrix decomposition \citep{Alabert, Davis1987}, Gibbs sampling \citep{GibbsSRF, Galli, Lantu2012}, discrete \citep{chiles2009geostatistics, Dietrich, pardo, wood1994simulation} or continuous \citep{Arroyo2020, emery2016, Lantuejoul, Shinozuka1971} Fourier approaches. Covariance (positive semidefinite) kernels also have important applications in various other branches of mathematics, such as numerical analysis, scientific computing and machine learning, where the use of compactly-supported kernels yields sparse Gram matrices and implies an important gain in storage and computation.

The expression \eqref{hygeo9} bears a resemblance to the Buhmann covariance kernels \citep{buhmann1998, Buhmann}, to the generalized Wendland covariance kernels \citep{Bevilacqua2020, Bevilacqua2019, Gneiting2002, zastavnyi2006some} and to the scale mixtures of Wendland kernels defined by  \cite{porcu2013}, all of which are also compactly supported. Our proposal, nevertheless, escapes from these three families: on the one hand, Buhmann's integral cannot yield the kernel \eqref{hygeo9} due to the restrictions on its parameters (the integrand contains a term $(1-t^\delta)$ with $\delta \leq \frac{1}{2}$ instead of $\delta=1$ in our case). On the other hand, the definition of the generalized Wendland kernel uses a different expression of the integrand, with a $t^2$ instead of a $t$ in one of the factors; a similar situation occurs for Porcu's mixtures of Wendland kernels, which use $\| \boldsymbol h \|$ instead of $\| \boldsymbol h \|^2$ in the integrand. We will see, however, that the family of generalized Wendland covariances is included in the Gauss hypergeometric class of covariance kernels (Section \ref{wendland}). Other compactly-supported covariance kernels involving the hypergeometric function ${}_2F_1$ have been proposed by \cite{porcu2013} and \cite{Porcu2014}, but none coincides with \eqref{2F1}.\\

The previously defined nonnegativity and integrability conditions yield the following restrictions on the parameters to provide a valid univariate covariance kernel.

\begin{theorem}[Parameter space]
\label{theorem1}
The Gauss hypergeometric covariance \eqref{2F1} is a valid covariance kernel in $\mathbb{R}^d$ and, consequently, its spectral density \eqref{spectraldensity2} is nonnegative and integrable, if the following sufficient conditions hold:
\begin{itemize}
    \item $a > 0$;
    \item $\alpha > \frac{d}{2}$;
    \item $2(\beta-\alpha)(\gamma-\alpha) \geq \alpha$;
    \item $2(\beta+\gamma) \geq 6\alpha + 1$.
\end{itemize}
\end{theorem}
In the following, $\mathcal{P}_d$ denotes the set of triplets $(\alpha,\beta,\gamma)$ of $\mathbb{R}_+^3$ satisfying the last three conditions of Theorem \ref{theorem1} (in passing, this notation is consistent with the previous definition of $\mathcal{P}_0$) and $\mathcal{G}_d$ denotes the set of kernels of the form $\sigma^2 G_d(\cdot; a,\alpha,\beta,\gamma)$ with $\sigma>0$, $a>0$ and $(\alpha,\beta,\gamma) \in \mathcal{P}_d$.  
These kernels are compactly supported, being identically zero outside the ball of radius $a$. Also note that $\mathcal{P}_{d} \subsetneq \mathcal{P}_{d^\prime}$ for any $d>d^\prime\geq0$.

\subsection{Main properties}

\begin{theorem} [Positive definiteness]
\label{positivity}
The $d$-dimensional Gauss hypergeometric covariance kernel \eqref{2F1} is positive definite, not just semidefinite, in $\mathbb{R}^d$.
\end{theorem}

\begin{theorem} [Restriction to subspaces]
\label{restriction}
The restriction of the $d$-dimensional Gauss hypergeometric covariance kernel \eqref{2F1} to any subspace $\mathbb{R}^{d-k}$, $k \in \{0,\hdots,d-1\}$, belongs to the family of Gauss hypergeometric covariance kernels $\mathcal{G}_{d-k}$.
\end{theorem}

\begin{theorem} [Extension to higher-dimensional spaces]
\label{extension}
The extension of the $d$-dimensional Gauss hypergeometric covariance kernel \eqref{2F1} to a higher-dimensional space $\mathbb{R}^{d+k}$, $k \in \mathbb{N}$, belongs to the family of Gauss hypergeometric covariance kernels $\mathcal{G}_{d+k}$ provided that $(\alpha+\frac{k}{2},\beta+\frac{k}{2},\gamma+\frac{k}{2}) \in \mathcal{P}_{d+k}$. 
\end{theorem}

\begin{remark}
For any set of finite parameters $(\alpha, \beta,\gamma) \in \mathcal{P}_{d}$, there exists a finite nonnegative integer $k$ such that $(\alpha+\frac{k}{2},\beta+\frac{k}{2},\gamma+\frac{k}{2}) \in \mathcal{P}_{d+k}$ and $(\alpha+\frac{k+1}{2},\beta+\frac{k+1}{2},\gamma+\frac{k+1}{2}) \not \in \mathcal{P}_{d+k+1}$: the extension of the Gauss hypergeometric covariance kernel with parameters $(\alpha, \beta,\gamma)$ in spaces of dimension greater than $d+k$ is no longer a valid covariance kernel. This agrees with Schoenberg's theorem \citep{Schoenberg}, according to which an isotropic function is a positive semidefinite kernel in Euclidean spaces of any dimension if, and only if, it is a nonnegative mixture of Gaussian covariance kernels, which the Gauss hypergeometric covariance (as any compactly supported kernel) is not.
\end{remark}

\begin{theorem}[Continuity and smoothness]
\label{continuity}
The function $(r,a,\alpha,\beta,\gamma) \mapsto g_d(r; a,\alpha,\beta,\gamma)$ from $\mathbb{R}_+ \times \mathbb{R}_+^* \times \mathcal{P}_d$ to $\mathbb{R}$ is
\begin{itemize}
    \item continuous with respect to $r$ on $[0,+\infty[$ and infinitely differentiable on $]0,a[$ and $]a,+\infty[$;
    \item continuous and infinitely differentiable with respect to $a$ on $]0,r[$ and $]r,+\infty[$;
    \item continuous and infinitely differentiable with respect to $\alpha$, $\beta$ and $\gamma$. 
\end{itemize}
\end{theorem}

\begin{theorem}[Differentiability at $r=a$]
\label{differentiabilitya}
The function $r \mapsto g_d(r; a,\alpha,\beta,\gamma)$ from $\mathbb{R}_+$ to $\mathbb{R}$ is $k$ times differentiable at $r=a$ if, and only if, $\beta - \alpha + \gamma > k+\frac{d}{2}+1$.
\end{theorem}

\begin{theorem}[Differentiability at $r=0$]
\label{differentiability0}
The function $r \mapsto g_d(r; a,\alpha,\beta,\gamma)$ from $\mathbb{R}_+$ to $\mathbb{R}$ is $k$ times differentiable at $r=0$ (therefore, it can be associated with a $\lfloor k/2 \rfloor$ times mean-square differentiable random field, with $\lfloor \cdot \rfloor$ denoting the floor function) if, and only if, $\alpha>\frac{k+d}{2}$.

\end{theorem}

\begin{theorem}[Monotonicity]
\label{monotonicity}
The function $(r,a,\alpha,\beta,\gamma) \mapsto g_d(r; a,\alpha,\beta,\gamma)$ from $\mathbb{R}_+ \times \mathbb{R}_+^* \times \mathcal{P}_d$ to $\mathbb{R}$ is 
\begin{itemize}
    \item decreasing in $r$ on $[0,a]$ and identically zero on $[a,+\infty[$;
    \item increasing in $a$ on $[r,+\infty[$ and identically zero on $]0,r]$;
    \item decreasing in $\beta$ if $0<r<a$, constant in $\beta$ if $r=0$ or if $r \geq a$;
    \item decreasing in $\gamma$ if $0<r<a$, constant in $\gamma$ if $r=0$ or if $r \geq a$.
\end{itemize}
\end{theorem}

\begin{theorem}[Mont\'ee]
\label{montee}
If $G_d(\cdot,a,\alpha,\beta,\gamma) \in \mathcal{G}_d$ and $\mathfrak{M}_k$ stands for the transitive upgrading (mont\'ee) of order $k$, $k \in \{0,\hdots,d-1\}$ (Appendix \ref{appendixA}), then $\mathfrak{M}_k (G_d(\cdot,a,\alpha,\beta,\gamma)) \in \mathcal{G}_{d-k}$ and its radial part is proportional to $g_{d}(\cdot,a,\alpha+\frac{k}{2},\beta+\frac{k}{2},\gamma+\frac{k}{2})$. In other words, when looking at the radial part of the covariance kernel, the mont\'ee of order $k$ amounts to upgrading the $\alpha$, $\beta$ and $\gamma$ parameters by $\frac{k}{2}$.

\end{theorem}

\begin{theorem} [Descente]
\label{descente}
If $G_d(\cdot,a,\alpha,\beta,\gamma) \in \mathcal{G}_d$ and $k \in \mathbb{N}$, then $\mathfrak{M}_{-k }(G_d(\cdot,a,\alpha,\beta,\gamma)) \in \mathcal{G}_{d+k}$ and its radial part is proportional to $g_{d}(\cdot,a,\alpha-\frac{k}{2},\beta-\frac{k}{2},\gamma-\frac{k}{2})$, provided that $(\alpha-\frac{k}{2},\beta-\frac{k}{2},\gamma-\frac{k}{2}) \in \mathcal{P}_{d+k}$.
\end{theorem}

\begin{remark}
Theorems \ref{differentiabilitya}, \ref{differentiability0}, \ref{montee} and \ref{descente} show that a mont\'ee (descente) of order $2k$ increases (decreases) the differentiability order by $2k$ near the origin, but only by $k$ near the range. 
\end{remark}

\begin{remark}
Compare the mont\'ee, descente, restriction and extension operations in Theorems \ref{restriction}, \ref{extension}, \ref{montee} and \ref{descente}. Both the extension and mont\'ee of order $k$ upgrade the parameters $\alpha, \beta$ and $\gamma$ by $\frac{k}{2}$, but the latter reduces the space dimension by $k$ whereas the former increases the dimension. Conversely, the restriction and descente of order $k$ downgrade the parameters $\alpha, \beta$ and $\gamma$ by $\frac{k}{2}$, but the latter increases the dimension by $k$ whereas the former reduces the dimension.
\end{remark}

\subsection{Examples}

\subsubsection{Euclid's hat (spherical) covariance kernel}

For $\alpha>0$, $\beta=\alpha+\frac{1}{2}$ and $\gamma=2\alpha$, the generalized hypergeometric function ${}_1 F_2$ can be expressed in terms of a squared Bessel function \citep{erdelyi1953higher}:
\begin{equation}
\label{J2}    
  {}_1 F_2\left( \alpha;\alpha+\frac{1}{2},2\alpha; -{(\pi a \|{\boldsymbol u}\|)^2} \right)  =  \Gamma^2\left(\alpha+\frac{1}{2}\right)  \left( \frac{\pi a \|{\boldsymbol u}\|}{2} \right)^{1-2\alpha} J^2_{\alpha-\frac{1}{2}}(\pi a \|{\boldsymbol u}\|). 
  \end{equation}
Equations \eqref{spectraldensity2} and \eqref{J2}, together with the Legendre duplication formula for the gamma function \citep[formula 5.5.5]{olver2010nist} yield the following result, valid for $\boldsymbol{u} \in \mathbb{R}^d$ and $\kappa \in \mathbb{N}$:
\begin{equation*}
\label{sph1}
\begin{split}
    \widetilde{G}_d &\left({\boldsymbol u}; a,\frac{d+1}{2}+\kappa,  \frac{d}{2}+1+\kappa,d+1+2\kappa \right) \\
    &=   \frac{\Gamma(\kappa+1) \Gamma(\frac{d}{2}+1+2\kappa) \Gamma^2(\frac{d}{2}+1)}{\pi^\frac{d-1}{2} \Gamma(\kappa+\frac{1}{2}) \Gamma^2(\frac{d}{2}+1+\kappa) 2^{2\kappa}  \|{\boldsymbol u}\|^{d}}  J_{\frac{d}{2}+\kappa}^2(\pi a \|{\boldsymbol u}\|).
    \end{split}
\end{equation*}
One recognizes the spectral density of the mont\'ee of order $2\kappa$ of the spherical covariance in $\mathbb{R}^d$ \citep{Arroyo2020}. The case $\kappa=0$ corresponds to the $d$-dimensional spherical covariance (triangular or tent covariance in $\mathbb{R}$, circular covariance in $\mathbb{R}^2$, usual spherical covariance in $\mathbb{R}^3$, pentaspherical in $\mathbb{R}^5$)
\citep[formula II.5.2]{matheron1965variables}, also known as Euclid's hat \citep{Gneiting1999}, while the cases $\kappa=1$ and $\kappa=2$ correspond to the $d$-dimensional cubic and penta covariances, respectively \citep{chiles2009geostatistics}. 
Interestingly, these spherical and upgraded spherical kernels can be extended to parameters that are not integer or half-integer by taking $\alpha>\frac{d}{2}, \beta=\alpha+\frac{1}{2}$ and $\gamma=2\alpha$ (i.e., $\kappa \not \in \mathbb{N}$). Such extended kernels correspond to the so-called fractional mont\'ee (if $\alpha>\frac{d+1}{2}$) or fractional descente (if $\frac{d}{2}<\alpha<\frac{d+1}{2}$) of the $d$-dimensional spherical covariance kernel \citep{matheron1965variables, Gneiting2002}.

\subsubsection{Generalized Wendland and Askey covariance kernels}
\label{wendland}

The generalized Wendland covariance in $\mathbb{R}^d$ with range $a>0$ and smoothness parameter $\kappa>0$ is defined as:
\begin{equation*}
\label{wend1}
 {\boldsymbol h} \mapsto \frac{\Gamma(\ell+2\kappa+1)}{ \Gamma(\ell+1) \Gamma(2\kappa)}  \int_0^1 t (1-t)^{\ell} \left( t^2 - \frac{\|{\boldsymbol h}\|^2}{a^2} \right)_+^{\kappa - 1} \text{d}t, 
\end{equation*}
with $\ell \geq \frac{d+1}{2}+\kappa$. \cite{Bevilacqua2020}, \cite{Chernih}, \cite{Hubbert} and \cite{zastavnyi2006some} showed that this covariance and its spectral density can be written under the forms \eqref{2F1} and \eqref{spectraldensity2}, respectively, with $\alpha = \frac{d+1}{2}+\kappa$, $\beta = \frac{d+\ell+1}{2}+\kappa$ and $\gamma = \frac{d+\ell}{2}+1+\kappa$. The cases when $\ell=\lfloor \frac{d}{2}+\kappa \rfloor+1$ and $\kappa$ is an integer or a half-integer yield the original \citep{Wendland} and missing \citep{Schaback} Wendland functions, respectively. The radial parts of the former are truncated polynomials in $[0,a]$, while that of the latter involve polynomials, logarithms and square root components \citep{Chernih}. 

The above parameterization with $\kappa=0$, i.e., $\alpha = \frac{d+1}{2}$, $\beta = \frac{d+\ell+1}{2}$ and $\gamma = \frac{d+\ell}{2}+1$, yields the well-known Askey covariance \citep{Askey}, the expression of which can be recovered by using Equation \eqref{2F1} along with formula 15.4.17 of \cite{olver2010nist}:
\begin{equation*}
\label{ask}
G_d \left( {\boldsymbol h}; a,\frac{d+1}{2},\frac{d+\ell+1}{2},\frac{d+\ell}{2}+1 \right) = \left( 1-\frac{\|\mathbf{h}\|}{a} \right)_{+}^{\ell}, \quad \boldsymbol h \in \mathbb{R}^d, \, \ell \geq \frac{d+1}{2}.
\end{equation*}
In spaces of even dimension, the lower bound $\frac{d+1}{2}$ for $\ell$ is less than the one $\lfloor \frac{d}{2}\rfloor +1$ found by \cite{Askey} and agrees with the findings of \cite{Gasper}.

\subsubsection{Truncated power expansions and truncated polynomial covariance kernels}

The Gauss hypergeometric covariance reduces to a finite power expansion by choosing $\alpha-\frac{d}{2} \not \in \mathbb{N}$, $\beta-\frac{d}{2} = N \in \mathbb{N}$ and $\gamma-\alpha = M \in \mathbb{N}$. Using formula \eqref{2F1c1} in Appendix \ref{appendixA} and the duplication formula for the gamma function, one finds:
\begin{equation*}
\label{powerseries1}
\begin{split}
g_d&\left(r; a,\alpha,\frac{d}{2}+N,\alpha+M\right) \\
&= \frac{\Gamma(\frac{d}{2}-\alpha+1) \Gamma(N)}{\Gamma(\frac{d}{2}-\alpha-M+1)} \sum_{n=0}^{N-1} \frac{(-1)^n \Gamma(\frac{d}{2}-\alpha-M+1+n)}{ \Gamma(\frac{d}{2}-\alpha+1+n) \Gamma(N-n) \, n!} \left(\frac{r}{a}\right)^{2n} \\
&+\frac{\Gamma(\frac{d}{2}-\alpha+1) \Gamma(N)}{\Gamma(\frac{d}{2}-\alpha-M+1)} \sum_{n=0}^{M-1} \frac{(-1)^n \Gamma(\alpha-\frac{d}{2}-N+1+n) }
{\Gamma(\alpha-\frac{d}{2}+1+n) \Gamma(M-n) \, n!} \left(\frac{r}{a}\right)^{2n+2\alpha-d}, \quad 0 \leq r < a.
\end{split}
\end{equation*}

A similar expansion is found by choosing $\alpha-\frac{d}{2} \not \in \mathbb{N}$, $\gamma-\frac{d}{2} = N \in \mathbb{N}$ and $\beta-\alpha = M \in \mathbb{N}$.

In both cases, if $\alpha-\frac{d}{2}$ is a half-integer, the radial part of the covariance is a polynomial function, truncated at zero for $r>a$. The Askey and original Wendland kernels and, when the space dimension $d$ is an odd integer, the spherical kernels are particular cases of these truncated polynomial kernels.

\subsection{Asymptotic cases}

\begin{theorem}[uniform convergence to the Mat\'ern covariance kernel]
\label{theorem5}
Let $\alpha>\frac{d}{2}$. As $a$, $\beta$ and $\gamma$ tend to infinity such that $\frac{a}{2\sqrt{\beta \gamma}}$ tends to a positive constant $b$, the Gauss hypergeometric covariance converges uniformly on $\mathbb{R}^d$ to the Mat\'ern covariance with scale factor $b$ and smoothness parameter $\alpha-\frac{d}{2}$: 
\begin{equation}
\label{matern}
\boldsymbol{h} \mapsto \frac{2}{\Gamma(\alpha-\frac{d}{2})} \left( \frac{\|\boldsymbol{h}\|}{2b} \right)^{\alpha-\frac{d}{2}} {K}_{\alpha-\frac{d}{2}}\left( \frac{\|\boldsymbol{h}\|}{b} \right), \quad \boldsymbol{h} \in \mathbb{R}^d, 
\end{equation}
where $K_{\alpha-\frac{d}{2}}$ is the modified Bessel function of the second kind of order $\alpha-\frac{d}{2}$.
\end{theorem}

\begin{theorem}[uniform convergence to generalized Laguerre kernel]
\label{laguerre}
As $a$ and $\gamma$ tend to infinity in such a way that $\frac{a}{\sqrt{\gamma}}$ tends to a positive constant $b$, the Gauss hypergeometric
covariance converges uniformly on $\mathbb{R}^d$ to the covariance kernel
\begin{equation*}
\label{laguerre1}
\boldsymbol{h} \mapsto \frac{\Gamma(\beta-\frac{d}{2})}{\Gamma(\alpha-\frac{d}{2})} \, L\left(\frac{d}{2}-\beta+1,\frac{d}{2}-\alpha+1,\frac{\, \| \boldsymbol{h} \|^2}{b^2}\right), \quad \boldsymbol{h} \in \mathbb{R}^d,
\end{equation*}
where $L$ is the Laguerre function of the second kind, defined by \citep[formula D.7]{matheron1965variables}:
\begin{equation*}
\label{laguerre0}
L(\alpha,\beta,x) = \frac{1}{\Gamma(\beta-\alpha)} \int_1^{+\infty} \exp(-u\,x) u^{\alpha-1} (u-1)^{\beta-\alpha-1} du, \quad x \in \mathbb{R}_+, \beta>\alpha.
\end{equation*}
The same result holds by interchanging $\beta$ and $\gamma$.
\end{theorem}

\begin{theorem}[uniform convergence to Tricomi's confluent hypergeometric kernel]
\label{tricomi}
As $\alpha-\frac{d}{2}$ tends to a positive even integer $2n$ and $a$ and $\gamma$ tend to infinity such that $\frac{a}{\sqrt{\gamma}}$ tends to a positive constant $b$, the Gauss hypergeometric
covariance converges uniformly on $\mathbb{R}^d$ to the  covariance kernel
\begin{equation*}
\label{kummer}
\boldsymbol{h} \mapsto \frac{\Gamma(\frac{d}{2}-\beta+2n+1)}{\Gamma(2n)} \,
U\left(\frac{d}{2}-\beta+1,1-2n,-\frac{\| \boldsymbol{h} \|^2}{b^2}\right), \quad \boldsymbol{h} \in \mathbb{R}^d,
\end{equation*}
where $U$ is Tricomi's confluent hypergeometric function  \citep[formula 13.2.6]{olver2010nist}. The same result holds by interchanging $\beta$ and $\gamma$.
\end{theorem}

\begin{theorem}[uniform convergence to incomplete gamma kernel]
\label{incgamma}
As $a$ and $\gamma$ tend to infinity in such a way that $\frac{a}{\sqrt{\gamma}}$ tends to a positive constant $b$ and $\beta = \frac{d}{2}+1$, the Gauss hypergeometric
covariance converges uniformly on $\mathbb{R}^d$ to the  covariance kernel
\begin{equation*}
\label{power1}
\boldsymbol{h} \mapsto 
Q\left(\alpha-\frac{d}{2},\frac{\| \boldsymbol{h} \|^2}{b^2}\right), \quad \boldsymbol{h} \in \mathbb{R}^d,
\end{equation*}
where $Q$ is the regularized incomplete gamma function \citep[formula 8.2.4]{olver2010nist}. The same result holds by interchanging $\beta$ and $\gamma$.
\end{theorem}

\begin{remark}
If, furthermore, $\alpha = \frac{d+1}{2}$, one obtains the complementary error function $\erfc(\frac{\| \boldsymbol{h} \|}{b})$, which is positive semidefinite in $\mathbb{R}^d$ for any dimension $d$ \citep{Gneiting1999}. 
\end{remark}

\begin{theorem}[uniform convergence to the Gaussian kernel, part 1]
\label{theorem6}
As $a, \alpha, \beta, \gamma$ tend to infinity in such a way that $a \sqrt{\frac{\alpha}{\beta \gamma}}$ tends to a positive constant $b$, the Gauss hypergeometric covariance converges uniformly on $\mathbb{R}^d$ to the Gaussian covariance with scale factor $b$:
\begin{equation}
\label{gauss}
\boldsymbol{h} \mapsto \exp \left( -\frac{\|\boldsymbol{h}\|^2}{b^2} \right), \quad \boldsymbol{h} \in \mathbb{R}^d. 
\end{equation}
\end{theorem}

\begin{theorem}[uniform convergence to the Gaussian kernel, part 2]
\label{theorem6b}
As $\beta$ tends to $\alpha$ and $a$ and $\gamma$ tend to infinity in such a way that $(\alpha,\beta,\gamma) \in \mathcal{P}_d$ and $\frac{a}{\sqrt{\gamma}}$ tends to a positive constant $b$, the Gauss hypergeometric covariance converges uniformly on $\mathbb{R}^d$ to the Gaussian covariance with scale factor $b$. The same result holds by interchanging $\beta$ and $\gamma$.
\end{theorem}

\begin{remark}
All the previous asymptotic kernels are positive semidefinite in Euclidean spaces of any dimension $d$, as the parameters $(\alpha,\beta,\gamma)$ can belong to $\mathcal{P}_d$ for sufficiently large $\beta$ and/or $\gamma$ values.
\end{remark}

\section{Multivariate compactly-supported hypergeometric covariance kernels}
\label{multivariate}

Let $p$ be a positive integer and consider a $p \times p$ matrix-valued kernel as:
\begin{equation}
\label{hygeo13}
   {\boldsymbol G}_d({\boldsymbol{h}}; {\boldsymbol a},{\boldsymbol \alpha},{\boldsymbol \beta},{\boldsymbol \gamma}, \boldsymbol \rho) = [\rho_{ij} G_d({\boldsymbol h}; a_{ij},\alpha_{ij},\beta_{ij},\gamma_{ij})]_{i,j=1}^p, \quad \boldsymbol{h} \in \mathbb{R}^d,
\end{equation}
where ${\boldsymbol a} = [a_{ij}]_{i,j=1}^p$, ${\boldsymbol \alpha} =  [\alpha_{ij}]_{i,j=1}^p$, ${\boldsymbol \beta} =  [\beta_{ij}]_{i,j=1}^p$, ${\boldsymbol \gamma} =  [\gamma_{ij}]_{i,j=1}^p$ and ${\boldsymbol \rho} =  [\rho_{ij}]_{i,j=1}^p$ are symmetric real-valued matrices  of size $p \times p$. The following theorem establishes various sufficient conditions on these matrices for ${\boldsymbol G}_d({\boldsymbol{h}}; {\boldsymbol a},{\boldsymbol \alpha},{\boldsymbol \beta},{\boldsymbol \gamma}, \boldsymbol \rho)$ to be a valid matrix-valued covariance kernel in $\mathbb{R}^d$.

\begin{theorem}[Multivariate sufficient validity conditions]
\label{theorem2aa}
The $p$-variate Gauss hypergeometric kernel \eqref{hygeo13} is a valid matrix-valued covariance kernel in $\mathbb{R}^d$ if the following sufficient conditions hold (see the definitions of conditionally negative semidefinite matrices and multiply monotone functions in Appendix \ref{appendixA}):
\begin{itemize}
\item [(1).] 
\begin{itemize}
    \item[(i)] $\boldsymbol{a} = a \boldsymbol{1}$ with $a>0$;
    \item[(ii)] $\boldsymbol{\alpha} = \alpha \boldsymbol{1}$;
    \item[(iii)] $\boldsymbol{\beta}$ is symmetric and conditionally negative semidefinite;
    \item[(iv)] $\boldsymbol{\gamma}$ is symmetric and conditionally negative semidefinite;
    \item[(v)] $(\alpha,\beta_{ij},\gamma_{ij}) \in \mathcal{P}_{d}$ for all $i,j$ in $[1,\hdots,p]$;
    \item[(vi)] $(\alpha,\beta,\gamma) \in \mathcal{P}_{0}$, with $\beta < \beta_{ij}$ and $\gamma < \gamma_{ij}$ for all $i,j$ in $[1,\hdots,p]$;
    \item[(vii)] $\left[\frac{\rho_{ij} \Gamma(\beta_{ij}-\frac{d}{2}) \Gamma(\gamma_{ij}-\frac{d}{2})}{\Gamma(\beta_{ij}-\beta) \Gamma(\gamma_{ij}-\gamma)} \right]_{i,j=1}^p$ is symmetric and positive semidefinite;
\end{itemize}
\item [or]
\item [(2).] 
\begin{itemize}
    \item[(i)] $a_{ij} = \max\{\varepsilon_i, \varepsilon_j\}$ if $i \neq j$ and $a_{ii} = \varepsilon_i - \delta_i$, with $0 \leq \delta_i < \varepsilon_i$ for $i = 1,\hdots,p$;
    \item[(ii)] $\boldsymbol{\alpha} = \alpha \boldsymbol{1}$;
    \item[(iii)] $\boldsymbol{\beta}$ is symmetric and conditionally negative semidefinite;
    \item[(iv)] $\boldsymbol{\gamma}$ is symmetric and conditionally negative semidefinite;
    \item[(v)] $(\alpha,\beta_{ij},\gamma_{ij}) \in \mathcal{P}_{d}$ for all $i,j$ in $[1,\hdots,p]$;
    \item[(vi)] $(\alpha+1,\beta+1,\gamma+1) \in \mathcal{P}_0$;
    \item[(vii)] $\left[\frac{\rho_{ij} a_{ij}^d \Gamma(\beta_{ij}-\frac{d}{2}) \Gamma(\gamma_{ij}-\frac{d}{2})}{\Gamma(\beta_{ij}-\beta) \Gamma(\gamma_{ij}-\gamma)} \right]_{i,j=1}^p$ is symmetric and positive semidefinite;
\end{itemize}
\item [or]
\item [(3).] 
\begin{itemize}
    \item[(i)] $a_{ij}^2 = \psi_1(\| \boldsymbol{s}_i - \boldsymbol{s}_j \|)$, with $\psi_1$ a positive function in $\mathbb{R}_+$ that has a $(q+1)$-times monotone derivative, $q \in \mathbb{N}$ and $\boldsymbol{s}_1, \hdots, \boldsymbol{s}_p \in \mathbb{R}^{2q+1}$;
    \item[(ii)] $\boldsymbol{\alpha} = \alpha \boldsymbol{1}$;
    \item[(iii)] $\boldsymbol{\beta}$ is symmetric and conditionally negative semidefinite;
    \item[(iv)] $\boldsymbol{\gamma}$ is symmetric and conditionally negative semidefinite;
    \item[(v)] $(\alpha,\beta_{ij},\gamma_{ij}) \in \mathcal{P}_{d}$ for all $i,j$ in $[1,\hdots,p]$;
    \item[(vi)] $(\alpha+q+2,\beta+q+2,\gamma+q+2) \in \mathcal{P}_0$ for $q \in \mathbb{N}$;
    \item[(vii)] $\left[\frac{\rho_{ij} a_{ij}^d \Gamma(\beta_{ij}-\frac{d}{2}) \Gamma(\gamma_{ij}-\frac{d}{2})}{\Gamma(\beta_{ij}-\beta) \Gamma(\gamma_{ij}-\gamma)} \right]_{i,j=1}^p$ is symmetric and positive semidefinite;
\end{itemize}
\item [or]
\item [(4).] 
\begin{itemize}
    \item[(i)] $a_{ij} = a$ if $i \neq j$ and $a_{ii} = a - \delta_i$, with $0 \leq \delta_i < a$ for $i = 1,\hdots,p$;
    \item[(ii)] $\alpha_{ij} = \psi_2(\| \boldsymbol{t}_i - \boldsymbol{t}_j \|)$, with $\psi_2$ a function in $\mathbb{R}_+$ with values in $]0,\frac{2\gamma-1}{4}]$ and a $(q^{\prime}+1)$-times monotone derivative, $q^{\prime} \in \mathbb{N}$, $\gamma>\frac{1}{2}$ and $\boldsymbol{t}_1, \hdots, \boldsymbol{t}_p \in \mathbb{R}^{2q^{\prime}+1}$;
    \item[(iii)] $\boldsymbol{\beta} - \boldsymbol{\alpha} - \boldsymbol{1}$ is symmetric, conditionally negative semidefinite and with positive entries;
    \item[(iv)] $\boldsymbol{\gamma}$ is symmetric and conditionally negative semidefinite;
    \item[(v)] $\left[\frac{\rho_{ij}  a_{ij}^d \Gamma(\beta_{ij}-\frac{d}{2}) \Gamma(\gamma_{ij}-\frac{d}{2})}{\alpha_{ij} \Gamma(\alpha_{ij}-\frac{d}{2}) \Gamma(\beta_{ij}-\alpha_{ij}-1) \Gamma(\gamma_{ij}-\gamma)} \right]_{i,j=1}^p$ is symmetric and positive semidefinite;
\end{itemize}
\item [or]
\item [(5).] 
\begin{itemize}
    \item[(i)] $a_{ij}^2 = \psi_1(\| \boldsymbol{s}_i - \boldsymbol{s}_j \|)$, with $\psi_1$ a positive function that has a $(q+1)$-times monotone derivative, $q \in \mathbb{N}$ and $\boldsymbol{s}_1, \hdots, \boldsymbol{s}_p \in \mathbb{R}^{2q+1}$;
    \item[(ii)] $\alpha_{ij} = \psi_2(\| \boldsymbol{t}_i - \boldsymbol{t}_j \|)$, with $\psi_2$ a positive function in $\mathbb{R}_+$ with values in $]0,\frac{2\gamma-1}{4}]$ and a $(q^{\prime}+1)$-times monotone derivative, $q^{\prime} \in \mathbb{N}$, $\gamma>\frac{1}{2}$ and $\boldsymbol{t}_1, \hdots, \boldsymbol{t}_p \in \mathbb{R}^{2q^{\prime}+1}$;
    \item[(iii)] $\boldsymbol{\beta} - \boldsymbol{\alpha} - \boldsymbol{1}$ is symmetric, conditionally negative semidefinite and with positive entries;
    \item[(iv)] $\boldsymbol{\gamma}$ is symmetric and conditionally negative semidefinite;
    \item[(v)] $(\alpha_{ij}+q+3,\alpha_{ij}+q+4,\gamma+q+3) \in \mathcal{P}_0$ for all $i,j$ in $[1,\hdots,p]$;
    \item[(vi)] $\left[\frac{\rho_{ij}  a_{ij}^{d+2} \Gamma(\beta_{ij}-\frac{d}{2}) \Gamma(\gamma_{ij}-\frac{d}{2})}{\alpha_{ij} \Gamma(\alpha_{ij}-\frac{d}{2}) \Gamma(\beta_{ij}-\alpha_{ij}-1) \Gamma(\gamma_{ij}-\gamma)} \right]_{i,j=1}^p$ is symmetric and positive semidefinite.
\end{itemize}
\end{itemize}
The conditions derived by interchanging $\boldsymbol{\beta}$ and $\boldsymbol{\gamma}$ in (4) and (5) also lead to a valid covariance kernel.
\end{theorem}

\section{Specific bivariate compactly-supported hypergeometric covariance kernels}
\label{bivariate}

In addition to the general sufficient conditions established in Theorem \ref{theorem2aa}, one can obtain three specific bivariate kernels by satisfying the following determinantal inequality: 
\begin{equation*}
    \widetilde{G}_d({\boldsymbol u}; a_{11},\alpha_{11},\beta_{11},\gamma_{11}) \widetilde{G}_d({\boldsymbol u}; a_{22},\alpha_{22},\beta_{22},\gamma_{22}) \geq \rho_{12} \, \widetilde{G}_d^2({\boldsymbol u}; a_{12},\alpha_{12},\beta_{12},\gamma_{12}), \quad {\boldsymbol u} \in \mathbb{R}^d.
\end{equation*}

\begin{itemize}
    \item[(i)] 
    For $x>0, \alpha>0, \beta \in ]\alpha+\frac{1}{2},2\alpha]$, one has the following inequality \citep[Theorem 5.1]{cho2018}:
\begin{equation*}
    {}_1F_2\left(\alpha;\beta,3\alpha+\frac{1}{2}-\beta;-\frac{x^2}{4}\right) \geq \Gamma^2\left(\alpha+\frac{1}{2} \right) \left(\frac{x}{2}\right)^{1-2\alpha} J^2_{\alpha-\frac{1}{2}}\left(x\right).
\end{equation*}
This implies that a valid bivariate kernel can be obtained by putting:
$$ \boldsymbol{a} = \left[\begin{matrix}
a & a\\
a & a
\end{matrix}\right], \boldsymbol{\alpha} = \left[\begin{matrix}
\alpha & \alpha\\
\alpha & \alpha
\end{matrix}\right], \boldsymbol{\beta} = \left[\begin{matrix}
\beta_1 & \alpha+\frac{1}{2}\\
\alpha+\frac{1}{2} & \beta_2
\end{matrix}\right], \boldsymbol{\gamma} = \left[\begin{matrix}
3\alpha+\frac{1}{2}-\beta_1 & 2\alpha\\
2\alpha & 3\alpha+\frac{1}{2}-\beta_2
\end{matrix}\right], \boldsymbol{\rho} = \left[\begin{matrix}
1 & \rho\\
\rho & 1
\end{matrix}\right],$$
with $a>0$, $\alpha>\frac{d}{2}$, $\alpha+\frac{1}{2} < \beta_1 \leq 2\alpha$, $\alpha+\frac{1}{2} < \beta_2 \leq 2\alpha$ and
\begin{equation*}
\rho^2 \leq \frac{\Gamma^2(\alpha+\frac{1}{2}) \Gamma^2(2\alpha) \Gamma(\beta_1-\frac{d}{2}) \Gamma(\beta_2-\frac{d}{2}) \Gamma(3\alpha-\beta_1+\frac{1-d}{2}) \Gamma(3\alpha-\beta_2+\frac{1-d}{2})}{\Gamma^2(\alpha+\frac{1-d}{2}) \Gamma^2(2\alpha-\frac{d}{2}) \Gamma(\beta_1) \Gamma(\beta_2) \Gamma(3\alpha-\beta_1+\frac{1-d}{2}) \Gamma(3\alpha-\beta_2+\frac{1-d}{2})}.
\end{equation*}

\item[(ii)] The same line of reasoning applies with the inequality \citep[Theorem 5.2]{cho2018}:
\begin{equation*}
    {}_1F_2\left(\alpha;\beta,\alpha+\beta-\frac{1}{2};-\frac{x^2}{4}\right) \geq \Gamma^2\left(\beta \right) \left(\frac{x}{2}\right)^{2-2\beta} J^2_{\beta-1}\left(x\right), \quad x>0, \alpha>0, \beta \geq \alpha+\frac{1}{2}.
\end{equation*}
This implies the validity of the following kernel in $\mathbb{R}^d$:
$$ \boldsymbol{a} = \left[\begin{matrix}
a & a\\
a & a
\end{matrix}\right], \boldsymbol{\alpha} = \left[\begin{matrix}
\alpha_1 & \beta-\frac{1}{2}\\
\beta-\frac{1}{2} & \alpha_2
\end{matrix}\right], \boldsymbol{\beta} = \left[\begin{matrix}
\beta & \beta\\
\beta & \beta
\end{matrix}\right], \boldsymbol{\gamma} = \left[\begin{matrix}
\alpha_1+\beta-\frac{1}{2} & 2\beta-1\\
2\beta-1 & \alpha_2+\beta-\frac{1}{2}
\end{matrix}\right], \boldsymbol{\rho} = \left[\begin{matrix}
1 & \rho\\
\rho & 1
\end{matrix}\right],$$
with $a>0$, $\alpha_1>\frac{d}{2}$, $\alpha_2>\frac{d}{2}$, $\beta \geq \max\{\alpha_1,\alpha_2\}+\frac{1}{2}$ and 
$$
\rho^2\leq \frac{\Gamma(\alpha_1)\Gamma(\alpha_2) \Gamma(\alpha_1+\beta-\frac{d+1}{2})\Gamma(\alpha_2+\beta-\frac{d+1}{2}) \Gamma^2(\beta-\frac{d+1}{2}) \Gamma^2(2\beta-1)}{\Gamma(\alpha_1-\frac{d}{2})\Gamma(\alpha_2-\frac{d}{2}) \Gamma(\alpha_1+\beta-\frac{1}{2})\Gamma(\alpha_2+\beta-\frac{1}{2}) \Gamma^2(\beta-\frac{1}{2}) \Gamma^2(2\beta-1-\frac{d}{2})}.$$

\item[(iii)] Likewise, one has \citep[Theorem 5.3]{cho2018}:
\begin{equation*}
    {}_1F_2\left(\alpha;\beta,2\alpha;-\frac{x^2}{4}\right) \geq \Gamma^2\left(\beta \right) \left(\frac{x}{2}\right)^{2-2\beta} J^2_{\beta-1}\left(x\right), \quad x>0, \alpha>0, \beta \geq \alpha+\frac{1}{2}.
\end{equation*}
This implies the validity of the following kernel in $\mathbb{R}^d$:
$$ \boldsymbol{a} = \left[\begin{matrix}
a & a\\
a & a
\end{matrix}\right], \boldsymbol{\alpha} = \left[\begin{matrix}
\alpha_1 & \beta-\frac{1}{2}\\
\beta-\frac{1}{2} & \alpha_2
\end{matrix}\right], \boldsymbol{\beta} = \left[\begin{matrix}
\beta & \beta\\
\beta & \beta
\end{matrix}\right], \boldsymbol{\gamma} = \left[\begin{matrix}
2\alpha_1 & 2\beta-1\\
2\beta-1 & 2\alpha_2
\end{matrix}\right], \boldsymbol{\rho} = \left[\begin{matrix}
1 & \rho\\
\rho & 1
\end{matrix}\right],$$
with $a>0$, $\alpha_1>\frac{d}{2}$, $\alpha_2>\frac{d}{2}$, $\beta \geq \max\{\alpha_1,\alpha_2\}+\frac{1}{2}$ and 
$$
\rho^2\leq \frac{\Gamma(\alpha_1)\Gamma(\alpha_2) \Gamma(2\alpha_1-\frac{d}{2})\Gamma(2\alpha_2-\frac{d}{2}) \Gamma^2(\beta-\frac{d+1}{2}) \Gamma^2(2\beta-1)}{\Gamma(\alpha_1-\frac{d}{2})\Gamma(\alpha_2-\frac{d}{2}) \Gamma(2\alpha_1)\Gamma(2\alpha_2) \Gamma^2(\beta-\frac{1}{2}) \Gamma^2(2\beta-1-\frac{d}{2})}.$$
\end{itemize}

These kernels escape from the cases presented in Theorem \ref{theorem2aa}, insofar as $\boldsymbol{\beta}$ and $\boldsymbol{\gamma}$ are not conditionally semidefinite negative in kernel (i), $\boldsymbol{\alpha}$ is not proportional to the all-ones matrix in kernels (ii) and (iii), and $\boldsymbol{\beta}-\boldsymbol{\alpha}$ is not conditionally semidefinite negative in all three kernels. Interestingly, if $\alpha$ (kernel (i)) or $\beta$ (kernels (ii) and (iii)) is an integer or a half-integer, the cross-covariances (off-diagonal entries of $G_d$) are univariate spherical kernels, but the direct covariances (diagonal entries of $G_d$) are not, unless $\beta_1=\beta_2=2\alpha$ or $\alpha_1=\alpha_2=\beta-\frac{1}{2}$, respectively.

\section{Concluding remarks}
\label{conclusions}

The class of Gauss hypergeometric covariance kernels presented in this work includes the stationary univariate kernels that are most widely used in spatial statistics: spherical, Askey, generalized Wendland and, as asymptotic cases, Mat\'ern and Gaussian. Figure~\ref{fig:parameterspace} maps these kernels in the parameter space $\mathcal{P}_d$. 
Concerning multivariate covariance kernels, under Conditions (1) of Theorem \ref{theorem2aa},  $\boldsymbol{a}$ is proportional to the all-ones matrix, i.e., all the direct and cross-covariances share the same range. In contrast, Conditions (2) and (3) allow different ranges, at the price of additional restrictions on the shape parameters $\boldsymbol{\alpha}$, $\boldsymbol{\beta}$ and $\boldsymbol{\gamma}$ that exclude a few covariance kernels located on the boundary of the parameter space $\mathcal{P}_d$, such as the spherical and Askey kernels. Even more interesting, Conditions (4) and (5) allow both $\boldsymbol{a}$ and $\boldsymbol{\alpha}$ not to be proportional to the all-ones matrix, i.e., the direct and cross-covariances not to share the same range nor the same behavior at the origin. This versatility makes the proposed multivariate Gauss hypergeometric covariance kernel a compactly-supported competitor of the well-known multivariate Mat\'ern kernel \citep{Apanasovich}.

\begin{figure}
\centering
\includegraphics[width=0.7\textwidth]{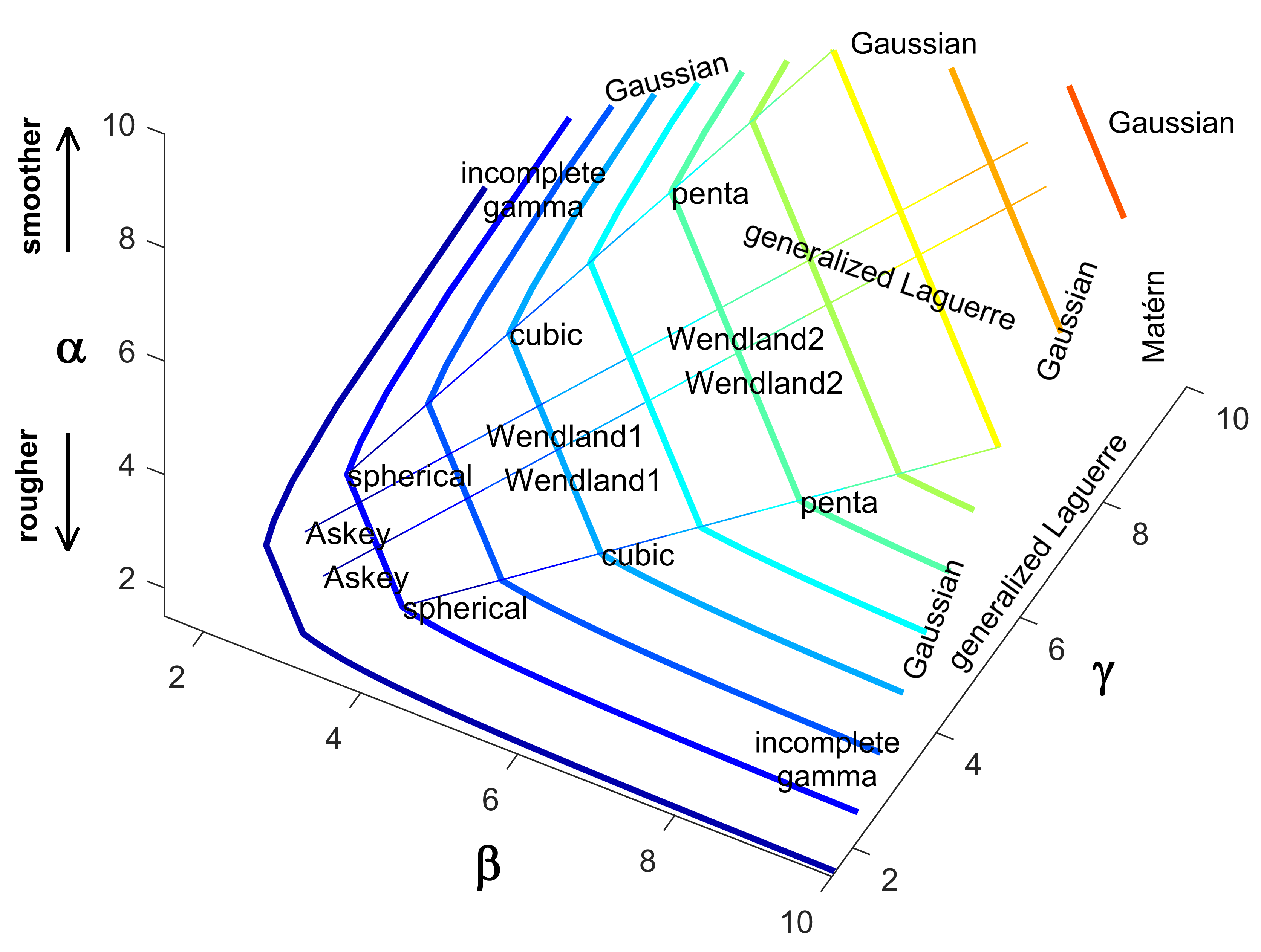}
\caption{Positioning of common covariance kernels in the parameter space $\mathcal{P}_d$ (with, here, $d=3$) of the hypergeometric covariance. The colored lines represent the upper boundary of $\mathcal{P}_3$, the color being a function of $\alpha$, from the lowest (blue) to the highest (red) values; the thin lines correspond to the spherical and Askey-Wendland families. The lower boundary of $\mathcal{P}_3$ is the white plane $\alpha=\frac{3}{2}$. The greater $\alpha$, the more regular the hypergeometric covariance at the origin.}
\label{fig:parameterspace}
\end{figure}

\begin{appendix}
\section*{Appendices}

\section{Technical definitions and lemmas}
\label{appendixA}

\begin{definition}[Mont\'ee and descente]
For $k \in \mathbb{N}$, $k<d$, the transitive upgrading or mont\'ee of order $k$ is the operator $\mathfrak{M}_k$ that transforms an isotropic covariance in $\mathbb{R}^d$ into a an isotropic covariance in $\mathbb{R}^{d-k}$ with the same radial spectral density \citep{matheron1965variables}. The reciprocal operator is the transitive downgrading (descente) of order $k$ and is denoted as $\mathfrak{M}_{-k}$.
\end{definition}

\begin{definition}[Conditionally negative semidefinite matrix]
\label{defvariog}
A $p \times p$ symmetric real-valued matrix $\boldsymbol{A}$ is conditionally negative semidefinite if, for any vector $\boldsymbol{\omega}$ in $\mathbb{R}^p$ whose components add to zero, one has $\boldsymbol{\omega}^{\top} \, \boldsymbol{A} \boldsymbol{\omega} \leq 0$.
\end{definition}

\begin{example}
\label{variogex}
Examples of conditionally negative semidefinite matrices include the all-ones matrix $\boldsymbol 1$ or the matrix $\boldsymbol A=[a_{ij}]_{i,j=1}^p$ with 
$$
a_{ij} = \frac{\eta_{i}+\eta_{j}}{2} + \psi(\boldsymbol{s}_i,\boldsymbol{s}_j),
$$
for any $\eta_{1}, \ldots, \eta_{p}$ in $\mathbb{R}$, $\boldsymbol{s}_1, \ldots, \boldsymbol{s}_p$ in $\mathbb{R}^d$, and variogram $\psi$ on $\mathbb{R}^d \times \mathbb{R}^d$ \citep{matheron1965variables, chiles2009geostatistics}. Also, the set of conditionally negative semidefinite matrices is a closed convex cone, so that the product of a conditionally negative semidefinite matrix with a nonnegative constant, the sum of two conditionally negative semidefinite matrices, or the limit of a convergent sequence of conditionally negative semidefinite matrices are still conditionally negative semidefinite.
\end{example}

\begin{lemma}[\cite{Berg}]
\label{superlemma}
A symmetric real-valued matrix $\boldsymbol{A}=[a_{ij}]_{i,j=1}^p$ is conditionally negative semidefinite if and only if $[\exp(- t \, a_{ij})]_{i,j=1}^p$ is positive semidefinite for all $t \geq 0$. 
\end{lemma}

\begin{definition}[Multiply monotone function]
For $q \in \mathbb{N}$, a $q$-times differentiable function $\varphi$ on $\mathbb{R}_+$ is $(q+2)$-times monotone if $(-1)^k \varphi^{(k)}$ is nonnegative, nonincreasing and convex for $k=0, \hdots,q$. A $1$-time monotone function is a nonnegative and nonincreasing function on $\mathbb{R}_+$ \citep{Williamson}.
\end{definition}

\begin{lemma}[\cite{Williamson}]
A $(q+2)$-times monotone function, $q \geq -1$, admits the expression 
\begin{equation}
\label{multiplymonotone}
\varphi(x) =  \int_0^{+\infty} (1 - t \, x)_+^{q+1} \,\nu (\textnormal{d}t), \qquad x \in \mathbb{R}_+, 
\end{equation}
where $\nu$ is a nonnegative measure.
\end{lemma}

\begin{example}
Examples of $(q+2)$-times monotone functions include the truncated power function $x \mapsto b+(1 - \frac{x}{a})_+^{\eta}$ with $a>0$, $b \geq 0$ and $\eta \geq q+1$, the completely monotone functions, and positive mixtures and products of such functions.  
\end{example}

\begin{lemma}
\label{cov1F2a}
Let $q \in \mathbb{N}$, $\alpha, \beta$, $\gamma \in \mathbb{R}_+^*$ and $\psi_1$ a positive function in $\mathbb{R}_+$ whose derivative is $(q+1)$-times monotone. Then, the function $\Phi_1: \mathbb{R}^{2q+1} \to \mathbb{R}$ defined by
\begin{equation}
    \label{covalemma1a}
    \Phi_1({\boldsymbol x}) =  {}_1 F_2\left( \alpha;\beta,\gamma; -\psi_1(\| {\boldsymbol x} \|) \right), \quad {\boldsymbol x} \in \mathbb{R}^{2q+1},
\end{equation}
is a stationary isotropic covariance kernel in $\mathbb{R}^{2q+1}$ if $(\alpha+q+2,\beta+q+2,\gamma+q+2) \in \mathcal{P}_0$.
\end{lemma}

\begin{example}
Examples of functions $\psi_1$ satisfying the conditions of Lemma \ref{cov1F2a} include the integrated truncated power function $\psi_1(x)= b \, x + c-(1 - \frac{x}{a})_+^{\eta+1}$ ($a>0$, $b>0$, $c > 1$ and $\eta \geq q$) and the Bernstein functions (positive primitives of completely monotone functions), e.g. \citep{slgsogvdk10}:
\begin{itemize}
    \item $\psi_1(x) = 1+\log\left(1+\frac{x}{b}\right)$ with $b > 0$;
    \item $\psi_1(x) = \left(1+b \, x^\eta \right)^\theta$ with $b > 0$, $\eta \in ]0,1]$ and $\theta \in ]0,1]$;
    \item $\psi_1(x) = 1+x \, (x+b)^{-\eta}$ with $b > 0$ and $\eta \in ]0,1]$.
\end{itemize}

\end{example}

\begin{lemma}
\label{cov1F2g}
Let $q^{\prime} \in \mathbb{N}$, $\gamma > 0$, $x > 0$ and $\psi_2$ a positive function in $\mathbb{R}_+$ upper bounded by $\alpha_{\max}=\frac{2\gamma-1}{4}$ and whose derivative is $(q^{\prime}+1)$-times monotone. Then, the function $\Phi_2: \mathbb{R}^{2q^{\prime}+1} \to \mathbb{R}$ defined by
\begin{equation}
    \label{covalemma1g}
    \Phi_2(\boldsymbol{y}) =  {}_1 F_2\left(\psi_2(\| \boldsymbol{y}\|);\psi_2(\| \boldsymbol{y}\|)+1,\gamma;-x \right), \quad \boldsymbol{y} \in \mathbb{R}^{2q^{\prime}+1},
\end{equation}
is a stationary isotropic covariance kernel in $\mathbb{R}^{2q^{\prime}+1}$.
\end{lemma}

\begin{lemma}
\label{cov1F2z}
Let $q, q^{\prime} \in \mathbb{N}$, $\gamma > 0$, $\psi_1$ a positive function in $\mathbb{R}_+$ with a $(q+1)$-times monotone derivative, and $\psi_2$ a positive function in $\mathbb{R}_+$ upper bounded by $\alpha_{\max}=\frac{2\gamma-1}{4}$ and with a $(q^{\prime}+1)$-times monotone derivative.
Then, the function $\Phi: \mathbb{R}^{2q+1} \times \mathbb{R}^{2q^{\prime}+1} \to \mathbb{R}$ defined by
\begin{equation}
    \label{covalemma1z}
    \Phi(\boldsymbol{x},\boldsymbol{y}) = \frac{1}{\psi_1(\| \boldsymbol{x}\|)} {}_1 F_2\left(\psi_2(\| \boldsymbol{y}\|);\psi_2(\| \boldsymbol{y}\|)+1,\gamma;-\psi_1(\| \boldsymbol{x}\|) \right), \quad \boldsymbol{x} \in \mathbb{R}^{2q+1}, \boldsymbol{y} \in \mathbb{R}^{2q^{\prime}+1},
\end{equation}
is positive semidefinite in $\mathbb{R}^{2q+1} \times \mathbb{R}^{2q^{\prime}+1}$ if $(\alpha+q+3,\alpha+q+4,\gamma+q+3) \in \mathcal{P}_0$.
\end{lemma}

\section{Proofs}
\label{appendixB}

\begin{proof}[Proof of Theorem \ref{positivity}]
Let $(\alpha,\beta,\gamma) \in \mathcal{P}_d$.
As the complex extension of the generalized hypergeometric function $x \mapsto {}_1F_2(\alpha,\beta,\gamma,x)$, $x \in \mathbb{C}$, is an entire function not identically equal to zero, its zeroes (if they exist) are isolated. It follows that there exists an nonempty open interval $I \subseteq \mathbb{R}$ such that ${}_1F_2(\alpha,\beta,\gamma,x)$ does not vanish, hence is positive, for all $x \in I$. Accordingly, the support of the spectral density \eqref{spectraldensity2} contains a nonempty open set of $\mathbb{R}^d$, which implies that the associated covariance kernel is positive
definite \cite{Dolloff}.
\end{proof}

\begin{proof}[Proof of Theorem \ref{restriction}]
The claim stems from the fact that $g_{d-k}(\cdot; a,\alpha-\frac{k}{2},\beta-\frac{k}{2},\gamma-\frac{k}{2})$ is the same as $g_d(\cdot; a,\alpha,\beta,\gamma)$ and that $(\alpha-\frac{k}{2},\beta-\frac{k}{2},\gamma-\frac{k}{2}) \in \mathcal{P}_{d-k}$ as soon as $(\alpha,\beta,\gamma) \in \mathcal{P}_{d}$.
\end{proof}

\begin{proof}[Proof of Theorem \ref{extension}]
The proof is analog to that of Theorem \ref{restriction}, with the additional restriction to ensure that the extended covariance remains valid in $\mathbb{R}^{d+k}$. 
\end{proof}

\begin{proof}[Proof of Theorem \ref{continuity}]
The continuity and differentiability with respect to $r$ stem from the fact that the Gauss hypergeometric function $x \mapsto {}_2F_1(a_1,a_2;b_1;x)$ with $b_1-a_1-a_2>0$ is continuous on the interval $[0,1]$, equal to $1$ at $x=0$, and infinitely differentiable on $]0,1[$. One deduces the continuity and differentiability with respect to $a$ by noting that, for fixed $\alpha$, $\beta$ and $\gamma$, $g_d(r; a,\alpha,\beta,\gamma)$ only depends on $\frac{r}{a}$. Finally, the continuity and differentiability with respect to $\alpha$, $\beta$ and $\gamma$ stem from the fact that the exponential function of base $\left( 1-(\frac{r}{a})^2\right)_+$ and the gamma function are infinitely differentiable wherever they are defined, and the hypergeometric function ${}_2F_1$ is an entire function of its parameters. 
\end{proof}

\begin{proof}[Proof of Theorem \ref{differentiabilitya}]
From \eqref{2F1}, it is seen that $r \mapsto g_d(r; a,\alpha,\beta,\gamma)$ is of the order of  $(1-\frac{r}{a})^{\beta +\gamma - \alpha -\frac{d}{2}-1}$ as $r \to a^-$, while it is identically zero for $r \to a^+$. Hence, this function is $k$ times differentiable (with zero derivatives of order $1, 2, \hdots, k$) at $r=a$, if, and only if, $\beta - \alpha+\gamma > k+\frac{d}{2}+1$.
\end{proof}

\begin{proof}[Proof of Theorem \ref{differentiability0}]
Using formula E.2.3 of \cite{matheron1965variables}, one obtains, for $\alpha-\frac{d}{2} \not\in \mathbb{N}$:
\begin{equation}
\label{2F1c1}
\begin{split}
g_d(r; a,\alpha,\beta,\gamma)  &= {}_2F_1\left(\frac{d}{2}-\gamma+1,\frac{d}{2}-\beta+1;\frac{d}{2}-\alpha+1;\frac{r^2}{a^2} \right)\\
&+\frac{\Gamma(\frac{d}{2}-\alpha) \Gamma(\beta-\frac{d}{2}) \Gamma(\gamma-\frac{d}{2}) }{ \Gamma(\alpha-\frac{d}{2}) \Gamma(\beta-\alpha)\Gamma(\gamma-\alpha)} \left(\frac{r}{a}\right)^{2\alpha-d} \\
& \times {}_2F_1\left(\alpha-\beta+1,\alpha-\gamma+1;\alpha-\frac{d}{2}+1;\frac{r^2}{a^2} \right), \quad 0 \leq r < a. 
\end{split}
\end{equation}
The right-hand side of \eqref{2F1c1} is a power series of $r^2$, plus a power series of $r^2$ (with a constant nonzero term) multiplied by $r^{2\alpha-d}$. Since $2\alpha-d$ is not an even integer, $r \mapsto g_d(r; a,\alpha,\beta,\gamma)$ turns out to be $k$ times differentiable at $r=0$ if, and only if, $\alpha>\frac{k+d}{2}$. If $\alpha-\frac{d}{2} \in \mathbb{N}$, then formula E.2.4 of \cite{matheron1965variables} shows that $r \mapsto g_d(r; a,\alpha,\beta,\gamma)$ is a power series of $r^2$ plus a power series of $r^2$ (with a constant nonzero term) multiplied by $r^{2\alpha-d} \log(\frac{r}{a})$, and the same conclusion prevails: $r \mapsto g_d(r; a,\alpha,\beta,\gamma)$ is $k$ times differentiable at $r=0$ if, and only if, $\alpha>\frac{k+d}{2}$.  
\end{proof}

\begin{proof}[Proof of Theorem \ref{monotonicity}]
Using an integral representation of the Gauss hypergeometric function ${}_2F_1$ \citep[formula 9.111]{grad}, the restriction of the radial function $g_d$
on the interval $[0,a]$ can be written as follows:
\begin{equation}
\label{gd}
\begin{split}
g_d(r; a,\alpha,\beta,\gamma)  = & \frac{\Gamma(\gamma-\frac{d}{2}) }{\Gamma(\gamma-\alpha) \Gamma(\alpha-\frac{d}{2})} \left( 1 - \frac{r^2}{a^2} \right)^{\beta - \alpha+\gamma -\frac{d}{2}-1} \\
&\times \int_0^1 t^{\gamma-\alpha-1} (1-t)^{\alpha-\frac{d}{2}-1} \left(1+\frac{t}{1-t} \frac{r^2}{a^2} \right)^{\alpha-\beta} \text{d}t, \quad r \in [0,a].
\end{split}
\end{equation}
Accordingly, on $[0,a]$, $r \mapsto g_d(r; a,\alpha,\beta,\gamma)$ appears as a beta mixture of powered quadratic functions of the form
$r \mapsto ( 1 - \frac{r^2}{a^2})^{\beta - \alpha+\gamma -\frac{d}{2}-1}$ multiplied by generalized Cauchy covariance functions of the form
$r \mapsto (1 + \frac{t}{1-t} \frac{r^2}{a^2})^{\alpha-\beta}$, with $t \in ]0,1[$, $a>0$ and $(\alpha,\beta,\gamma)\in \mathcal{P}_d$. Since all these functions are nonnegative and decreasing
on $[0,a]$, so is $r \mapsto g_d(r; a,\alpha,\beta,\gamma)$.


The monotonicity in $r$ implies the monotonicity in $a$, insofar as $g_d(r; a,\alpha,\beta,\gamma)$ only depends on $\frac{r}{a}$ for fixed $\alpha$,$\beta$ and $\gamma$.

Consider the integral representation \eqref{hygeo9} as a function of $r = \|\boldsymbol{h}\|$, $a$, $\alpha$, $\beta$ and $\gamma$. Based on the dominated convergence theorem, this function can be differentiated under the integral sign with respect to parameter $\gamma$, which leads to:
\begin{equation*}
\label{hygeo9c}
\begin{split}
  \frac{\partial g_d(r; a,\alpha,\beta,\gamma)}{\partial \gamma}  &=  \frac{ \Gamma(\beta - \frac{d}{2})}{ \Gamma(\alpha-\frac{d}{2}) \Gamma(\beta-\alpha)} \\
  & \times \int_{0}^1 t^{\alpha-\gamma} (1-t)_+^{\beta-\alpha-1} \left( t - \left(\frac{r}{a}\right)^2 \right)_+^{\gamma - \frac{d}{2} - 1} \ln\left(1-\frac{r^2}{t a^2} \right)_+ \text{d}t.  
\end{split}
\end{equation*}
This equation includes the $r=0$ instance, as $\gamma \mapsto g_d(1; a,\alpha,\beta,\gamma)$ is identically equal to $1$. The partial derivative is therefore always negative (if $0<r<a$) or zero (if $r=0$ or $r>a$), implying that $\gamma \mapsto g_d(r; a,\alpha,\beta,\gamma)$ is decreasing or constant in $\gamma$, respectively. The same result holds by substituting $\beta$ for $\gamma$ owing to the symmetry of the ${}_2F_1$ function.
\end{proof}

\begin{proof}[Proof of Theorem \ref{montee}]
For $(\alpha,\beta,\gamma) \in \mathcal{P}_d$, the radial part of $\mathfrak{M}_k (G_d(\cdot,a,\alpha,\beta,\gamma))$ is the Hankel transform of order $d-k$ of $\widetilde{g}_d(\cdot,a,\alpha,\beta,\gamma)$. From \eqref{hankel}, one has
\begin{equation*}
\mathfrak{M}_k( G_d(\cdot,a,\alpha,\beta,\gamma)) = \frac{\zeta_d(a,\alpha,\beta,\gamma)}{\zeta_{d-k}(a,\alpha,\beta,\gamma)} G_{d-k}(\cdot,a,\alpha,\beta,\gamma).
\end{equation*}
Since $\mathcal{P}_d \subset \mathcal{P}_{d-k}$, it follows that $\mathfrak{M}_k (G_d(\cdot,a,\alpha,\beta,\gamma)) \in \mathcal{G}_{d-k}$, its radial part being 
\begin{equation*}
\frac{\zeta_d(a,\alpha,\beta,\gamma)}{\zeta_{d-k}(a,\alpha,\beta,\gamma)} g_{d-k}(\cdot,a,\alpha,\beta,\gamma) =  \frac{\zeta_d(a,\alpha,\beta,\gamma)}{\zeta_{d-k}(a,\alpha,\beta,\gamma)} \, g_{d}\left(\cdot,a,\alpha+\frac{k}{2},\beta+\frac{k}{2},\gamma+\frac{k}{2}\right).
\end{equation*}
\end{proof}

\begin{proof}[Proof of Theorem \ref{descente}]
The proof follows that of Theorem \ref{montee}. The condition $(\alpha-\frac{k}{2},\beta-\frac{k}{2},\gamma-\frac{k}{2}) \in \mathcal{P}_{d+k}$ ensures that the downgraded covariance is positive semidefinite in $\mathbb{R}^{d+k}$, based on Theorem \ref{theorem1}.
\end{proof}

\begin{proof}[Proof of Theorem \ref{theorem5}]
The proof relies on  expansion \eqref{2F1c1} of the radial function $r \mapsto g_d(r; a,\alpha,\beta,\gamma)$, valid for $r \in [0,a]$ and $\alpha-\frac{d}{2} \not \in \mathbb{N}$.
Using formulae 5.5.3 and 5.11.12 of \cite{olver2010nist}, as well as the theorem of dominated convergence to interchange limits and infinite summations, one finds the following asymptotic equivalence:
\begin{equation*}
\label{2F1d}
\begin{split}
g_d(r; a,\alpha,\beta,\gamma) &\sim  {}_0F_1\left( ;\frac{d}{2}-\alpha+1;\frac{\beta \gamma r^2}{a^2} \right) \\ &+\frac{\Gamma(\frac{d}{2}-\alpha)}{ \Gamma(\alpha-\frac{d}{2})} \left(\frac{\beta \gamma r^2}{a^2}\right)^{\alpha-\frac{d}{2}} {}_0F_1\left( ;\alpha-\frac{d}{2}+1;\frac{\beta \gamma r^2}{a^2} \right), \quad r \leq a,
\end{split}
\end{equation*}
as $\beta \to +\infty$ and $\gamma \to +\infty$. The left-hand side can be expressed in terms of modified Bessel functions of the first ($I_{\eta}$) and second ($K_{\eta}$) kinds thanks to formulae 5.5.3, 10.27.4 and 10.39.9 of \cite{olver2010nist}, which finally yields:
\begin{equation}
\label{2F1e}
\begin{split}
g_d(r; a,\alpha,\beta,\gamma) &\sim \Gamma\left(\frac{d}{2}-\alpha+1\right) \left(\frac{\sqrt{\beta \gamma} r}{a}\right)^{\alpha-\frac{d}{2}} I_{\frac{d}{2}-\alpha}\left(\frac{2\sqrt{\beta \gamma} r}{a}\right) \\
&+\frac{\Gamma(\frac{d}{2}-\alpha)}{ \Gamma(\alpha-\frac{d}{2})} \Gamma\left(\alpha-\frac{d}{2}+1\right) \left(\frac{\sqrt{\beta \gamma} r}{a}\right)^{\alpha-\frac{d}{2}} I_{\alpha-\frac{d}{2}}\left(\frac{2\sqrt{\beta \gamma} r}{a}\right) \\
& = 
\frac{2} {\Gamma\left(\alpha-\frac{d}{2}\right)} \left(\frac{\sqrt{\beta \gamma} r}{a}\right)^{\alpha-\frac{d}{2}} K_{\alpha-\frac{d}{2}}\left(\frac{2\sqrt{\beta \gamma} r}{a}\right), \quad r \leq a.
\end{split}
\end{equation} 
Accordingly, $g_d(\cdot; a,\alpha,\beta,\gamma)$ tends pointwise to the radial part of the Mat\'ern covariance \eqref{matern} by letting $\beta$ and $\gamma$ tend to infinity and $a$ be asymptotically equivalent to $2 b \sqrt{\beta \gamma}$. In particular, since $a$ tends to infinity, the pointwise convergence is true for any $r\geq 0$.
It is also true if $\alpha-\frac{d}{2} \in \mathbb{N}$, as it suffices to consider the asymptotic equivalence \eqref{2F1e} with $\alpha-\delta-\frac{d}{2}$ and $\delta>0$ and then to let $\delta$ tend to zero, both the Gauss hypergeometric and Mat\'ern covariances being continuous with respect to the parameter $\alpha$. Note that the conditions of Theorem \ref{theorem1} are fulfilled when $\alpha$ is fixed and greater than $\frac{d}{2}$ and $\beta$ and $\gamma$ become infinitely large, so that $g_d(\cdot; a,\alpha,\beta,\gamma)$ in \eqref{2F1e} is the radial part of a valid covariance kernel.
Finally, because $g_d(\cdot; a,\alpha,\beta,\gamma)$ is a decreasing function on any compact segment of $\mathbb{R}_+$ for sufficiently large $a$ and $\beta$ or $\gamma$ (Theorem \ref{monotonicity}) and the limit function (the radial part of the Mat\'ern covariance \eqref{matern}) is continuous on $\mathbb{R}_+$, Dini's second theorem implies that the pointwise convergence is actually uniform on any compact segment of $\mathbb{R}_+$. In turn, since all the functions are lower bounded by zero, uniform convergence on a compact segment of $\mathbb{R}_+$ implies uniform convergence on $\mathbb{R}_+$.
\end{proof}

The proofs of Theorems \ref{laguerre} to \ref{theorem6b} use of the same argument as above to identify pointwise convergence with uniform convergence. This argument will be omitted for the sake of brevity.

\begin{proof}[Proof of Theorem \ref{laguerre}]
The starting point is the expansion \eqref{2F1c1} of $g_d(\cdot; a,\alpha,\beta,\gamma)$ in $[0,a]$.
Using formulae 5.5.3 and 5.11.12 of \cite{olver2010nist} and the dominated convergence theorem to interchange limits and infinite summations, one finds the following asymptotic equivalence as $\gamma$ tends to infinity:
\begin{equation}
\label{laguerre2}
\begin{split}
g_d&\left(r; a,\alpha,\beta,\gamma\right) \sim {}_1F_1\left(\frac{d}{2}-\beta+1;\frac{d}{2}-\alpha+1;-\frac{\gamma \, r^2}{a^2}\right) \\
&+ {}_1F_1\left(\alpha-\beta+1;\alpha-\frac{d}{2}+1;-\frac{\gamma \, r^2}{a^2}\right) \frac{\Gamma(\frac{d}{2}-\alpha) \Gamma(\beta-\frac{d}{2})}{\Gamma(\alpha-\frac{d}{2}) \Gamma(\beta-\alpha)} \left(\frac{\gamma \, r^2}{a^2}\right)^{\alpha-\frac{d}{2}}, \quad 0 \leq r < a.
\end{split}
\end{equation}
Using formula D.8 of \cite{matheron1965variables} and letting $a \to +\infty$ such that $\frac{a}{\sqrt{\gamma}} \to b > 0$ yields the claim.
\end{proof} 

\begin{proof}[Proof of Theorem \ref{tricomi}]
The proof relies on \eqref{laguerre2} and formulae 5.5.3 and 13.2.42 of \cite{olver2010nist}.
\end{proof}

\begin{proof}[Proof of Theorem \ref{incgamma}]
The proof relies on \eqref{laguerre2} and formulae 8.2.3, 8.2.4, 8.5.1 and 13.6.3 of \cite{olver2010nist}.
\end{proof} 

\begin{proof}[Proof of Theorem \ref{theorem6}]
The proof follows from Theorem \ref{theorem5} and the fact that the Mat\'ern covariance
\eqref{matern} with scale parameter $b/(2 \sqrt{\alpha})$ and smoothness parameter $\alpha$ tends to the Gaussian covariance \eqref{gauss} as $\alpha \to +\infty$. 
Following \cite{Chernih}, the convergence can also be shown by noting that the spectral density \eqref{spectraldensity2} of the Gauss hypergeometric covariance is asymptotically equivalent to
\begin{equation*}
\label{spectraldensity2gauss}
\begin{split}
    \widetilde{G}_d({\boldsymbol u}; a,\alpha,\beta,\gamma) &\sim   \left(\frac{\pi a^2 \alpha }{\beta \gamma}\right)^{\frac{d}{2}} \, 
    \sum_{n=0}^{+\infty} \frac{1}{n!} \left(-\frac{\alpha (\pi a \|{\boldsymbol u}\|)^2}{\beta \gamma} \right)^n \\
    &= \left(\frac{\pi a^2 \alpha }{\beta \gamma}\right)^{\frac{d}{2}} \, \exp\left(-\frac{\alpha (\pi a \|{\boldsymbol u}\|)^2}{\beta \gamma} \right), \quad {\boldsymbol u} \in \mathbb{R}^d,
\end{split}
\end{equation*}
as $\alpha \to +\infty$, $\beta \to +\infty$ and $\gamma \to +\infty$. If, furthermore, $a \to +\infty$ such that $a \sqrt{\frac{\alpha}{\beta \gamma}} \to b > 0$, then one obtains:
\begin{equation*}
\label{spectraldensity3gauss}
    \widetilde{G}_d({\boldsymbol u}; a,\alpha,\beta,\gamma) \sim   \pi^{\frac{d}{2}} b^d \, \exp\left(-(\pi b \|{\boldsymbol u}\|)^2 \right), \quad {\boldsymbol u} \in \mathbb{R}^d,
\end{equation*}
which coincides with the spectral density of the Gaussian covariance \eqref{gauss} \citep{Arroyo2020, Lantuejoul}.
\end{proof}

\begin{proof}[Proof of Theorem \ref{theorem6b}]
The proof follows from the asymptotic equivalence \eqref{laguerre2} for $\gamma$ tending to infinity. As $\beta$ tends to $\alpha$ and $a$ tends to infinity in such a way that $\frac{a}{\sqrt{\gamma}}$ tends to $b>0$, the first term in the right-hand side of \eqref{laguerre2} tends to $\exp(-r^2/b^2)$ and the second term to zero.
\end{proof}

\begin{proof}[Proof of Lemma \ref{cov1F2a}]
One has $\Phi_1({\boldsymbol x}) = \varphi_1 \circ \psi_1(\|\boldsymbol{x}\|)$, where $\varphi_1: x \mapsto {}_1 F_2\left( \alpha;\beta,\gamma; -x\right)$ is an infinitely differentiable function on $\mathbb{R}_+$, with \citep[formula 16.3.1]{olver2010nist}
\begin{equation*}
    (-1)^k \frac{\partial^k \varphi_1}{\partial x^k}(x) = \frac{\Gamma(\alpha+k) \Gamma(\beta) \Gamma(\gamma)}{\Gamma(\alpha) \Gamma(\beta+k) \Gamma(\gamma+k)} {}_1 F_2\left( \alpha+k;\beta+k,\gamma+k; -x\right), \quad x \in \mathbb{R}_+, k \in \mathbb{N}.
\end{equation*}
If $(\alpha+q+2,\beta+q+2,\gamma+q+2) \in \mathcal{P}_0$, then, for any $k=0, \hdots, q+2$, $(\alpha+k,\beta+k,\gamma+k) \in \mathcal{P}_0$ and $(-1)^k \frac{\partial^k \varphi_1}{\partial x^k}$ is nonnegative on $\mathbb{R}_+$, hence $\varphi_1$ is $(q+2)$-times monotone. Since $\psi_1$ is positive and has a $(q+1)$-times monotone derivative, the composite function $\varphi_1 \circ \psi_1$ is $(q+2)$-times monotone \citep[proposition 4.5]{Gneiting1999}. The fact that this composite function is continuous implies that $\Phi_1$ is positive semidefinite in $\mathbb{R}^{2q+1}$ \citep[criterion 1.3]{Askey, Micchelli, Gneiting1999}. 
\end{proof}

\begin{proof}[Proof of Lemma \ref{cov1F2g}]
$\Phi_2({\boldsymbol y}) = \varphi_2 \circ \psi_2(\|\boldsymbol{y}\|)$, where $\varphi_2: \alpha \mapsto {}_1 F_2\left(\alpha;\alpha+1,\gamma; -x\right)$ is a nonnegative function on $[0,\alpha_{\max}]$, insofar as $(\alpha,\alpha+1,\gamma) \in \mathcal{P}_0$ as soon as $\alpha \leq \alpha_{\max}$. This function is infinitely differentiable; for $k \in \mathbb{N}^*$, its $k$-th derivative, obtained with a term-by-term differentiation of \eqref{pFq}, is
\begin{equation}
\label{derivativealpha1F2}
\begin{split}
   \frac{\partial^k \varphi_2}{\partial \alpha^k}(\alpha) &= \sum_{n=1}^{+\infty} \frac{(-1)^{k-1} k! \, n \, \Gamma(\gamma)}{n! (\alpha+n)^{k+1} \Gamma(\gamma+n)} \left(-x\right)^n \\
&= \frac{(-1)^{k} k! \, x}{(\alpha+1)^{k+1} \gamma} \sum_{n=0}^{+\infty} \frac{(\alpha+1)^{k+1} \Gamma(\gamma+1)}{n! (\alpha+1+n)^{k+1} \Gamma(\gamma+1+n)} \left(-x\right)^n \\
&= \frac{(-1)^{k} k! \, x}{(\alpha+1)^{k+1} \gamma} \, {}_{k+1}F_{k+2}\left(\alpha+1,\hdots,\alpha+1;\alpha+2,\hdots,\alpha+2,\gamma+1;-x\right).
\end{split}
\end{equation}
If $\alpha \in [0,\alpha_{\max}]$,
then $(\alpha+1,\alpha+2,\gamma+1) \in \mathcal{P}_0$ and ${}_{k+1}F_{k+2}(\alpha+1,\hdots,\alpha+1;\alpha+2,\hdots,\alpha+2,\gamma+1;-x)$
is nonnegative, as a beta mixture of nonnegative ${}_{1}F_{2}$ functions \citep[formula 16.5.2]{olver2010nist}, which implies that $\varphi_2$ is completely monotone on $[0,\alpha_{\max}]$.
Since $\psi_2$ is positive with values in $[0,\alpha_{\max}]$ and has a $(q^{\prime}+1)$-times monotone derivative, the composition $\varphi_2 \circ \psi_2$ is $(q^{\prime}+2)$-times monotone on $\mathbb{R}_+$.
As it is continuous, this entails that $\Phi_2$ is a positive semidefinite in $\mathbb{R}^{2q^{\prime}+1}$ \citep{Micchelli}.
\end{proof}

\begin{proof}[Proof of Lemma \ref{cov1F2z}]
Let introduce $\Phi(\boldsymbol{x},\boldsymbol{y}) = \varphi(\psi_1(\|\boldsymbol{x}\|),\psi_2(\|\boldsymbol{y}\|))$,
where $\varphi: (x,\alpha) \mapsto \frac{1}{x} \, {}_1 F_2\left(\alpha;\alpha+1,\gamma; -x\right)$ is nonnegative and infinitely differentiable on $\mathbb{R}_+^* \times [0,\alpha_{\max}]$.
From \eqref{derivativealpha1F2}, it comes, for $k, k^\prime \in \mathbb{N}$:
\begin{equation*}
\label{derivativexalpha1F2}
\begin{split}
 &\frac{\partial^{k+k^{\prime}} \varphi}{\partial x^k \, \partial \alpha^{k^{\prime}}}(x,\alpha) = \sum_{n=0}^{+\infty} \frac{(-1)^{k+k^{\prime}} k^{\prime}! \, \Gamma(\gamma)}{\Gamma(n+1) (\alpha+n+k+1)^{k^{\prime}+1} \Gamma(\gamma+n+k+1)} \left(-x\right)^{n} \\
&= \frac{(-1)^{k+k^{\prime}} k^{\prime}! \, \Gamma(\gamma)}{(\alpha+k+1)^{k^{\prime}+1} \Gamma(\gamma+k+1)} \\
& \quad \times {}_{k^{\prime}+1}F_{k^{\prime}+2}\left(\alpha+k+1,\hdots,\alpha+k+1;\alpha+k+2,\hdots,\alpha+k+2,\gamma+k+1;-x \right).
\end{split}
\end{equation*}
If $(\alpha+q+3,\alpha+q+4,\gamma+q+3) \in \mathcal{P}_0$, then, for any $k=0, \hdots, q+2$, $(\alpha+k+1,\alpha+k+2,\gamma+k+1) \in \mathcal{P}_0$ and the hypergeometric term ${}_{k^{\prime}+1}F_{k^{\prime}+2}$ is nonnegative, as a beta mixture of nonnegative ${}_{1}F_{2}$ terms. Under this condition, $(-1)^{k+k^{\prime}} \frac{\partial^{k+k^{\prime}} \varphi}{\partial x^k \, \partial \alpha^{k^{\prime}}}$ is nonnegative for $k=0, \hdots, q+2$ and any  $k^\prime \in \mathbb{N}$. Accordingly, $\varphi$ is a bivariate multiply monotone function of order $(q+2,q^\prime+2)$, and so is the composite function $\varphi(\psi_1,\psi_2)$ \citep[proposition 4.5]{Gneiting1999}.
Arguments in \cite{Williamson} generalized to functions of two variables imply that $\varphi(\psi_1,\psi_2)$ is a mixture of products of truncated power functions of the form \eqref{multiplymonotone} (one function of $x$ with power exponent $q+1$ times one function of $\alpha$ with power exponent $q^\prime+1$) and is the radial part of a product covariance kernel in $\mathbb{R}^{2q+1} \times \mathbb{R}^{2q^{\prime}+1}$.
\end{proof}

\begin{proof}[Proof of Theorem \ref{theorem2aa}]
We start proving (1). Conditions (i), (ii) and (v) imply the existence of a spectral density associated with each direct or cross covariance (Theorem \ref{theorem1}). Based on Cram\'er's criterion \citep{cramer, chiles2009geostatistics},  $\boldsymbol{\widetilde{G}}_d(\cdot; a\boldsymbol{1},\alpha \boldsymbol{1},\boldsymbol{\beta},\boldsymbol{\gamma},\boldsymbol{\rho})$ is a valid matrix-valued spectral density function if, and only if, $\boldsymbol{\widetilde{G}}_d(\boldsymbol{u}; a \boldsymbol{1},\alpha \boldsymbol{1},\boldsymbol{\beta},\boldsymbol{\gamma},\boldsymbol{\rho})$ is positive semidefinite for any vector $\boldsymbol{u} \in \mathbb{R}^d$. The key of the proof is to expand this matrix as a positive mixture of positive semidefinite matrices. Such an expansion rests on the following identity, which can be obtained by a term-by-term integration of the infinite series \eqref{pFq} defining the generalized hypergeometric function ${}_1F_2$ along with formula 3.251.1 of \cite{grad}:
\begin{equation}
\label{expansion2aa}
\begin{split}
    \int_0^1 & \int_0^1  {}_1F_2  \left(\alpha;\beta,\gamma;-t_1 t_2 (a\,x)^2\right) t_1^{\beta-1}(1-t_1)^{\beta_{ij}-\beta-1}  t_2^{\gamma-1}(1-t_2)^{\gamma_{ij}-\gamma-1} \text{d}t_1 \text{d}t_2 \\
    &= \frac{\Gamma(\beta) \Gamma(\beta_{ij}-\beta) \Gamma(\gamma) \Gamma(\gamma_{ij}-\gamma)}{\Gamma(\beta_{ij}) \Gamma(\gamma_{ij})} {}_1F_2\left(\alpha;\beta_{ij},\gamma_{ij};-(a\,x)^2\right), 
\end{split}
\end{equation}
for $x \geq 0, a>0, \alpha>0,  \beta_{ij}>\beta>0$ and $\gamma_{ij}>\gamma>0$. Accordingly, for $\boldsymbol{u} \in \mathbb{R}^d$:
\begin{equation*}
\label{expansion1baa}
\begin{split}
    &{\widetilde{\boldsymbol{G}}}_d(\boldsymbol{u}; {a}\boldsymbol{1},{\alpha}\boldsymbol{1},\boldsymbol{\beta},\boldsymbol{\gamma},\boldsymbol{\rho}) = \frac{\pi^{\frac{d}{2}} a^{d} \Gamma(\alpha)   \Gamma(\boldsymbol{\beta} - \frac{d}{2}) \Gamma(\boldsymbol{\gamma} - \frac{d}{2}) \boldsymbol{\rho}}{\Gamma(\alpha-\frac{d}{2}) \Gamma(\beta) \Gamma(\boldsymbol{\beta}-\beta) \Gamma(\gamma) \Gamma(\boldsymbol{\gamma}-\gamma)} \\
    &\times \int_0^1 \int_0^1 {}_1F_2\left(\alpha;\beta,\gamma;-{t_1 \, t_2 (\pi a \|\boldsymbol{u}\|)^2}\right) t_1^{\beta-1}(1-t_1)^{\boldsymbol{\beta}-\beta-1} t_2^{\gamma-1}(1-t_2)^{\boldsymbol{\gamma}-\gamma-1} \text{d}t_1 \text{d}t_2,
\end{split}
\end{equation*}
with the products, quotients and powers taken element-wise. ${}_1F_2\left(\alpha;\beta,\gamma;-{t_1 t_2 (\pi a \|\boldsymbol{u}\|)^2}\right)$ is nonnegative for any $t_1, t_2 \in [0,1]$ under Condition (vi) \citep{cho2020rational}. Under Conditions (iii) and (iv), $(1-t_1)^{\boldsymbol{\beta}}$ and $(1-t_2)^{\boldsymbol{\gamma}}$ are positive semidefinite matrices (Lemma \ref{superlemma}). Along with Condition (vii), ${\widetilde{\boldsymbol{G}}}_d(\boldsymbol{u}; a \boldsymbol{1},\alpha \boldsymbol{1},\boldsymbol{\beta},\boldsymbol{\gamma},\boldsymbol{\rho})$ is positive semidefinite for any $\boldsymbol{u}$ in $\mathbb{R}^d$, as the elewent-wise product of positive semidefinite matrices, which completes the proof for (1).

We now prove (2). Under Condition (vi), the generalized hypergeometric function ${}_1F_2(\alpha;\beta,\gamma,x)$ is positive and increasing in $x$ on $\mathbb{R}$ \citep[formula 16.3.1]{olver2010nist}. Therefore, if $\boldsymbol{a}$ fulfills Condition (i), $[{}_1F_2\left(\alpha;\beta,\gamma;-t_1 t_2(\,a_{ij}\, x)^2\right)]_{i,j=1}^p$ is positive semidefinite, as the sum of a \emph{min} matrix with positive entries \citep[problem 7.1.P18]{Horn} and a diagonal matrix with nonnegative entries. The proof of (1) can then be adapted in a straightforward manner, by substituting such a positive semidefinite matrix for the positive scalar ${}_1F_2\left(\alpha;\beta,\gamma;-t_1 t_2(\,a\, x)^2\right)$.

The proof of (3) follows that of (2) and relies on the fact that, under Conditions (i) and (vi), the matrix $[{}_1F_2\left(\alpha;\beta,\gamma;-t_1 t_2 (\,a_{ij}\, x)^2\right)]_{i,j=1}^p$ is positive semidefinite for any $t_1$, $t_2$ and $x$ (Lemma \ref{cov1F2a}).

The proof of (4) is similar to that of (1), with \eqref{expansion2aa} replaced by
\begin{equation*}
\label{expansion23}
\begin{split}
    \int_0^1 & \int_0^1  {}_1F_2  \left(\alpha_{ij};\alpha_{ij}+1,\gamma;-t_1 t_2 (a_{ij}\,x)^2\right) t_1^{\alpha_{ij}}(1-t_1)^{\beta_{ij}-\alpha_{ij}-2}  t_2^{\gamma-1}(1-t_2)^{\gamma_{ij}-\gamma-1} \text{d}t_1 \text{d}t_2 \\
    &= \frac{\Gamma(\alpha_{ij}+1) \Gamma(\beta_{ij}-\alpha_{ij}-1) \Gamma(\gamma) \Gamma(\gamma_{ij}-\gamma)}{\Gamma(\beta_{ij}) \Gamma(\gamma_{ij})} {}_1F_2\left(\alpha_{ij};\beta_{ij},\gamma_{ij};-(a_{ij}\,x)^2\right),
\end{split}
\end{equation*}
for $x \geq 0, a_{ij}>0, \beta_{ij}-1>\alpha_{ij}>0$ and $\gamma_{ij}>\gamma>0$ for $i,j$ in $[1,\hdots,p]$.
Under Condition (ii), the composite function $t \mapsto \exp(-x(\psi_2(t)-\psi_2(0)))$ is $(q^{\prime}+2)$-times monotone  \citep[proposition 4.5]{Gneiting1999}, hence it is a mixture of truncated power functions of the form \eqref{multiplymonotone} and is the radial part of a positive semidefinite function in $\mathbb{R}^{2q^{\prime}+1}$ for any $x>0$.
A classical result by \cite{Schoenberg} states that $\boldsymbol{x} \mapsto \psi_2(\|\boldsymbol{x}\|)-\psi_2(0)$ is a variogram in $\mathbb{R}^{2q^{\prime}+1}$, so  $\boldsymbol{\alpha}$ is
conditionally negative semidefinite (Example \ref{variogex}) and $[t_1^{\alpha_{ij}}]_{i,j=1}^p$ is positive semidefinite for any $t_1 \in [0,1]$ (Lemma \ref{superlemma}). 
Under Conditions (iii) and (iv), $[(1-t_1)^{\beta_{ij}-\alpha_{ij}}]_{i,j=1}^p$ and $[(1-t_2)^{\gamma_{ij}}]_{i,j=1}^p$ are positive semidefinite for any $t_1, t_2 \in [0,1]$ (Lemma \ref{superlemma}).
Under Condition (ii), $[{}_1F_2  (\alpha_{ij};\alpha_{ij}+1,\gamma;-t_1 t_2 (a\,x)^2)]_{i,j=1}^p$ is also positive semidefinite for any $t_1, t_2 \in [0,1]$, $a>0$, $x>0$ (Lemma \ref{cov1F2g}). As $(\alpha_{ij}+1,\alpha_{ij}+2,\gamma+1) \in \mathcal{P}_0$, the generic entry of this matrix decreases with $a$ \citep[formula 16.3.1]{olver2010nist}, hence the matrix $[{}_1F_2  (\alpha_{ij};\alpha_{ij}+1,\gamma;-t_1 t_2 (a_{ij}\,x)^2)]_{i,j=1}^p$ has increased diagonal entries and is still positive semidefinite.
Finally, Condition (v) and Schur's product theorem imply that ${\widetilde{\boldsymbol{G}}}_d(\boldsymbol{u}; \boldsymbol{a},\boldsymbol{\alpha},\boldsymbol{\beta},\boldsymbol{\gamma},\boldsymbol{\rho})$ is
positive semidefinite for any $\boldsymbol{u}$ in $\mathbb{R}^d$, as the elewent-wise product of positive semidefinite matrices, which completes the proof of (4).

The proof of (5) follows the same line of reasoning as that of (4). The positive semidefiniteness of $[a_{ij}^{-2} \, {}_1F_2  (\alpha_{ij};\alpha_{ij}+1,\gamma;-t_1 t_2 (a_{ij}\,x)^2)]_{i,j=1}^p$ now stems from Conditions (i), (ii) and (v) together with Lemma \ref{cov1F2z}.

\end{proof}
\end{appendix}

\section*{Acknowledgements}
The authors acknowledge the funding of the National Agency for Research and Development of Chile, through grants ANID/FONDECYT/REGULAR/No.\ 1210050 (X. Emery and A. Alegr\'ia) and ANID PIA AFB180004 (X. Emery).


\bibliography{mybib}
\bibliographystyle{apalike}

\end{document}